\documentclass[12pt]{article}

\usepackage{graphics}
\usepackage{amsmath}
\usepackage{amsfonts}
\usepackage{amssymb}
\usepackage{graphicx}
\usepackage{appendix}
\usepackage{comment}
\input xy
\xyoption{all}
\usepackage{color}
\usepackage{graphicx,amssymb,amsfonts,latexsym,amsmath,amsthm}

\usepackage{amsfonts}
\usepackage{amssymb}
\usepackage{eucal}

\newtheorem{theorem}{Theorem}[section]
\newtheorem{lemma}[theorem]{Lemma}
\newtheorem{proposition}[theorem]{Proposition}
\newtheorem{definition}[theorem]{Definition}
\newtheorem{example}[theorem]{Example}

\newtheorem{corollary}[theorem]{Corollary}
\newtheorem{remark}[theorem]{Remark}

\newcommand{\cP}{{\cal P}}
\newcommand{\cM}{{\cal M}}

\def\p{\partial}

\def\cR{{\mathcal R}}
\def\cA{{\mathcal A}}
\def\cB{{\mathcal B}}

\def\cP{{\mathcal P}}

\def\di{\diamond}
\def\R{{ \mathbb{R}}}

\def\N{{\mathbb{N}}}
\def\I{{\mathbb{I}}}

\newcommand{\RE} {{\rm I \kern-2.8pt R} }

\newcommand{\beasnum}{\begin{eqnarray}}
\newcommand{\eeasnum}{\end{eqnarray}}
\newcommand{\beas}{\begin{eqnarray*}}
\newcommand{\eeas}{\end{eqnarray*}}
\begin{document}

\title{\bf {Selection-Mutation Differential Equations:\\
Long-Time Behavior of Measure-Valued Solutions}}

\author{Azmy S. Ackleh$^{\ddag}$, John Cleveland$^{\dag}$  and
Horst R. Thieme$^*$
\\[3mm]
$^\ddag$Department of Mathematics\\[-2mm]
University of Louisiana at Lafayette\\[-2mm]
Lafayette, Louisiana 70504-1010
\\[3mm]
$^\dag$Department of Mathematics\\[-2mm]
University of Wisconsin \\[-2mm]
Richland Center, WI  53581
\\[3mm]
$^*$School of Mathematical and Statistical Sciences
\\[-2mm]
Arizona State University
\\[-2mm]
Tempe, Arizona 85787-1804
}

\maketitle

\begin{abstract}
  We study the long-time behavior of solutions to a measure-valued selection-mutation model
  that we formulated in \cite{CLEVACK}.
  We establish permanence results for the full model,
   and we study the limiting behavior even when there is more than one strategy
   of a given fitness; a case that arises in applications.
    We show that for the pure selection case the solution of the dynamical system converges
 to a Dirac measure centered at the fittest strategy class provided that the support of the initial measure contains a fittest strategy;
 thus we term this Dirac measure an Asymptotically  Stable Strategy (ASS).
 We also show that when the strategy space is discrete,
 the selection-mutation model with small mutation has a locally asymptotically stable equilibrium that attracts all initial conditions that are positive at the fittest strategy.
 \\[3mm]
\end{abstract}


\section{ Introduction}
\label{sec:intro}

A significant part of evolutionary game theory (EGT) focuses on the creation and study of  mathematical models that
describe how the strategy profiles
in games change over time due to mutation and selection
(replication) \cite{Nowak,BrownVincent}.
In \cite{CLEVACK} we defined an evolutionary game theory (EGT) model as
an ordered triple $(Q,\mu(t),F(\mu(t)))$ subject to:
\begin{equation}
\label{mconstraint}
\frac{d}{dt}\mu(t)(E)=F(\mu(t))(E), \text{ for every}
~~E \in \mathcal{B}(Q).
\end{equation}
Here $Q$ is the strategy
(metric) space, $\mathcal{B}(Q)$ is the $\sigma$-algebra of  Borel subsets of $Q$, $\mu(t)$
is a time dependent family of  nonnegative finite Borel measures on $Q$,
and $F$ is a density dependent vector field such that $\mu$ and $F$
satisfy equation \eqref{mconstraint}. For a Borel subset $E$ of $Q$, $\mu(t) (E)$
denotes the measure $\mu(t)$ applied to $E$ {and can be interpreted as
the number of individuals at time $t$ that carry a strategy from the
set $E$.}

 We also formulated the following selection-mutation EGT model as a dynamical system
 $\phi(t,u, \gamma)$
 on the  state space of finite nonnegative Borel measures under the  weak$^*$ topology
 with $\mu(t) = \phi(t,u,\gamma)$ and $\bar \mu(t) = \mu(t)(Q)$:
\begin {equation}
\left\{\begin{array}{rl}
\label{M1}
 \displaystyle \frac{d}{dt}{\mu}(t)(E)
 = &  \displaystyle
  \int_Q {B}\big(\bar \mu(t),  q\big) \gamma( q)(E)\mu(t)(d  q)
\; -\; \int_E {D}\big(\bar \mu(t), q\big)
\mu(t)(d q) ,
\\ = & {F} (\mu(t), \gamma)(E),
\\
\mu(0)=& u.
\end{array}\right.
\end{equation}
Here $B(\bar \mu(t),q)$ and $D(\bar \mu(t),q)$ represent the
reproduction {(replication)} and mortality rates of individuals carrying strategy
$q$ when the total population size is $\bar \mu(t)= \mu(t)(Q)$.
 The probability kernel $\gamma(q)(E)$ represents the probability that an
individual carrying strategy $q$ produces {offspring} carrying strategies in the
Borel set $E$. We call $\gamma$ a {\em mutation kernel}.

The purpose of this paper is to complement
the well-posedness theory established in \cite{CLEVACK} with a study of the
long-time behavior of solutions to the model \eqref{M1}.
It is well known that the solutions of many such models constructed on the state space
of continuous or integrable functions converge to a Dirac measure
{concentrated} at the fittest
{strategy or trait} \cite{AFT,AMFH,calsina,CALCAD,GVA,P,GR1,GR2,ThiYang}. In \cite[ch.2]{P},
  these measure-valued limits are illustrated in a biological and
  adaptive dynamics environment. This convergence is in the weak$^* $ topology \cite{AMFH}.
   Thus, the asymptotic limit of the solution is not in such state spaces; it is a measure.
   Some models (e.g. \cite{AFT}, {\cite{ThiYang}}) have addressed this problem.
In \cite{AFT}, the authors formulated a {\it pure} selection model
{ on the space of finite signed measures with
   density dependent birth and mortality
functions and a 2-dimensional {strategy} space}.
They discussed existence-uniqueness of solutions and studied the long term behavior of
the model.
Here, we substantially generalize the results in that paper in several directions:
\begin{itemize}
\item [1)] We model selection and mutation,
and we allow for the selection-mutation kernel to be a family of measures
(thus simultaneously treating discrete and continuous
strategy spaces).
\item [2)] We consider more general nonlinearities in the rates, and thus
the results apply to a wider class of models.
    \item[3)] We
allow for more than one fittest strategy.
\end{itemize}

 To motivate our attempt of allowing more than one fittest strategy, recall
 that in  \cite{AMFH} the authors considered the following logistic growth with
 pure selection (i.e., strategies
replicate themselves exactly and no mutation occurs)  model:
\begin{equation}
 \frac{d}{dt} x(t,q) = x(t,q) (
q_1 -q_2 \bar x(t)), \label{logiseq}
\end{equation}
{where $\bar x(t) = \int_Q x(t,q) dq$} is the total population, $Q \subset \text{int}(\mathbb{R}_+^2)$ is a rectangle
 and the state space is the set of continuous real valued functions
 $C(Q)$. Each $ q=(q_1, q_2) \in Q$ is a pair where $q_1$ is an
 intrinsic replication rate and $q_2$ is an intrinsic mortality
 rate. The fittest strategy was defined as the one with the highest replication to mortality ratio, $\max_Q\{q_1/q_2\}$, which is unique for this choice of $Q$. Utilizing the uniqueness of the fittest strategy, the authors show that the solution converges
 to the Dirac mass centered at {the strategy} with this ratio.
 In Figure 1, we present two examples of strategy spaces $Q\subset \text{int}(\mathbb{R}_+^2)$:
{one that
is similar to that considered in \cite{AMFH} and has a unique fittest strategy  (left)
and another that has a continuum of fittest strategies (right).}

  \begin{figure}[htbp] \label{Fig1}
  \vspace*{-0.1in}
  \begin{center}
 \includegraphics[width=10cm, height=10cm]
{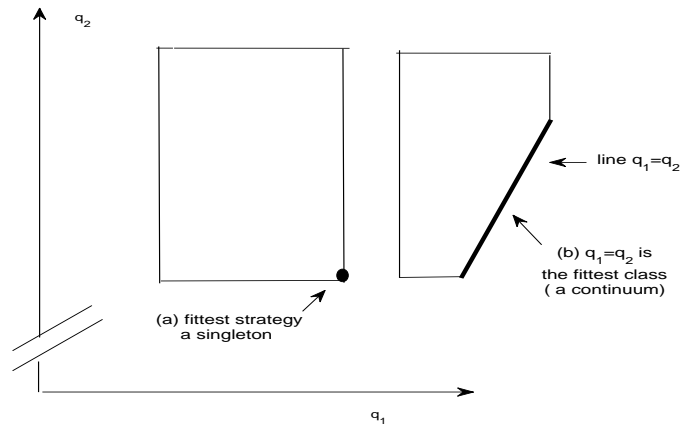}
 \end{center}
 \vspace*{-.9in}
 \caption{ Two examples of strategy spaces.}
\label{strategyspace}
\end{figure}

This paper is organized as follows.
In section \ref{sec:ass}, we provide some background material and make assumptions on the model parameters.  In section \ref{sec:asymp-gen}, we  provide some permanence and persistence results.
 In section \ref{sec:pure-sel}, we study the asymptotic behavior for the pure selection kernel, and in section \ref{sec:dir-mut}, we study the asymptotic behavior for directed mutation kernels. In section \ref{sec:disc}, we consider  discrete strategy spaces and establish asymptotic behavior  results for this case. In section \ref{sec:conclud},  we provide concluding remarks.


\section{Assumptions and background material}
\label{sec:ass}

In this section, we state assumptions and define notation that we will use throughout
 the paper, and we recall the main well-posedness result established in \cite{CLEVACK}.
 Here  $ \mathcal{M}= \mathcal{M}(Q)$  are the finite signed Borel measures on
 $Q$, a compact metric space.  $\mathcal{M}_{V,+}$ represents the positive cone under the total variation topology,
 and $\mathcal{M}_{w,+}$ represents the positive cone under the weak* topology.
 Let $\mathcal P_w= \mathcal P_w(Q) $ denote the probability
measures under the weak* topology and $
C(Q,\mathcal{P}_w)$ denote the continuous $\mathcal{P}_w$ valued functions on
$Q$ with the topology of uniform convergence.
The mutation kernel $\gamma$ is assumed to be an element of $C(Q,\mathcal{P}_w)$,
in other words, $\gamma $ has the Feller property {(see \cite{LaTh} and the references therein)}. Also, for any time dependent mapping, $f(t)$,  we let $f^\prime (t)=\frac{d}{dt} f(t)$.

\begin{lemma}
\label{re:measures-weak-top}
The following hold:
\begin{itemize}
\item $\cM_{w+}$ is a closed convex subset (actually cone) of the locally
convex topological vector space $\cM_w$ in which all bounded closed subsets
are compact.
\item There exists a norm $p$ on $\cM$ such that, for each bounded closed subset of
$\cM_w$, the induced metric space is complete and topologically equivalent to
the topological space induced by the weak$^*$ topology. The norm $p$ makes $\cM_+$ into
a complete metric space  which is topologically  equivalent to  $\cM_{w+}$.
\item $C(Q, \cP_w)$ is a complete convex subset of the normed vector space $C(Q, \cM_p)$ where $\cM_p$
is $\cM$ equipped with the norm $p$.
\end{itemize}
\end{lemma}

\begin{proof}
Since $\mathcal{M}_{V}$ is the norm dual of $C(Q)$,
the first statement is a consequence of the Alaoglu-Bourbaki theorem.

Since $Q$ is a compact metric space, its topology has
a countable base. So $C(Q)$ with the supremum norm is separable
(Theorem 7.6.3 in \cite{Bau}),
and so is its unit ball.
Choose a dense subset $S= \{f_k; k \in \N\}$ of the unit ball of $C(Q)$.
 For each $k\in \N$, define the
seminorm  $p_k$ on $\cM(Q)$ by
\begin{equation}
\label{eq:seminorm}
p_k ( \nu) = \Big | \int_Q f_k(q) \nu(dq)  \Big |.
\end{equation}
The following standard construction defines a norm  $p$ on $\cM(Q)$,
\begin{equation}
\label{eq:norm-weak}
p(\nu) = |\nu(Q)| + \sum_{k=1}^\infty 2^{-k} p_k (\nu).
\end{equation}
{This norm induces the weak$^*$ topology on every bounded subset of
$\cM_w$ {\cite[Thm.6.30]{ALI}} and on $\cM_{w+}$.
Further, any sequence that is a Cauchy sequence with respect to $p$ and
bounded in $\cM_w$ is a weak$^*$ Cauchy sequence and so converges
in the weak$^*$ topology and thus also with respect to $p$.
Any sequence in $\cM_+$ that is a Cauchy sequence with respect to $p$
is automatically bounded in total variation and converges with
respect to $p$ by the same arguments.}

The topology on $C(Q,\mathcal{P}_w)$ is induced by the norm
\begin{equation}
\label{eq:metric-mutation}
\|\gamma\| = \sup_{q \in Q} p ( \gamma(q)), \qquad { \gamma \in C(Q, \cM_p).}
\end{equation}
Standard considerations show that $C(Q,\mathcal{P}_w)$ is a closed
convex subset of the {normed vector space} $C(Q, \cM_p)$.
 \end{proof}

\begin{lemma}
\label{re:semicont2}
If $f: Q \to \R$  is bounded below (above)
and lower (upper) semicontinuous,
then the mapping
\[
\begin{split}
\nu & \mapsto \int f d\nu, \\
 \cal{M_+} & \to [-\infty, \infty],
\end{split}
\]
is lower (upper) semicontinuous.
\end{lemma}

\begin{proof} We will demonstrate the result for lower semicontinuous functions
that are bounded below;
the other case follows from this one by considering $-f$.
Let $\nu_m \rightarrow \nu$ in the weak* topology. Then there exists a sequence $ ( f_n)$ with each
$f_n $ being Lipschitz continuous such that $f_n(x) \uparrow f(x)$ for  $x \in X $
\cite[Thm.3.13]{ALI}. From
$$
\int f_n d \nu_m \le \int f d\nu_m \mbox { and } \int f_n d\nu_m
\stackrel{m \to \infty}{\longrightarrow} \int f_n d\nu,
$$
we see that $ \int f_n d\nu \le \liminf_{m} \int f d\nu_m $ for each $n.$
We use the Monotone Convergence Theorem and obtain
$$ \int f d\nu = \lim_{n \rightarrow \infty} \int f_n d\nu \le \liminf_{m} \int f d\nu_m . $$
By Lemma \ref{re:measures-weak-top},
this sequential characterization of lower semicontinuity
is equivalent to the topological one (inverse images of intervals $(b, \infty]$, $b \in \R$,
are open sets),
 and the function $ \nu \mapsto \int f d\nu $ is lower semicontinuous.
\end{proof}

\begin{lemma}
\label{re:semicont}
 Let $E$ be an open (closed) subset of the compact metric space $Q$.

\begin{itemize}
\item[(a)] Then the function $\rho: \cM_{w,+} \to \R_+$,
$\rho (\nu) = \nu(E)$, is lower (upper) semicontinuous.

\item[(b)] Also the function $\psi:C(Q,\mathcal{P}_w)\times Q \to \R_+$
  defined by $\psi(\gamma, q) = \gamma(q)(E)$ is lower (upper) semicontinuous.
\end{itemize}
\end{lemma}

\begin{proof}  We first notice that the characteristic function of an open
set is lower semicontinuous and the characteristic function of a closed set is upper semicontinuous.
If a set is both open and closed, then its characteristic function is continuous.
Hence (a) is immediate from Lemma {\ref{re:semicont2}}.
Likewise (b) follows once we notice
$$ C( Q, \cP_w) \times Q \rightarrow \cal{M}_{w+} \rightarrow \mathbb{R_+} $$
given by $$ ( \gamma, q) \mapsto \gamma(q) \mapsto \int f d \gamma(q) $$ is a
composition of a continuous and a lower (upper) semicontinuous function.
\end{proof}


\subsection{Birth and Mortality Rates}


Concerning the birth rate, $B(s,q)$, and the mortality rate, $
D(s,q)$,
{where $s \in [0,\infty)$ is the total population size and $q \in Q$
a strategy (trait)}, we make assumptions similar
to those used in \cite{AFT,CLEVACK}:
\begin{itemize}
\item[{\bf (A1)}] {$B: [0,\infty) \times Q \rightarrow [0,\infty)$ is continuous,
and $B(s,q)$    is locally Lipschitz
continuous in $s\ge 0$, uniformly with respect to $q \in Q$, and   nonincreasing in $s \ge 0$.}

\item[{\bf (A2)}] {$D: [0,\infty) \times Q \rightarrow [0,\infty)$ is
continuous, and $D(s,q)$
is
locally Lipschitz
continuous in $s \ge 0$, uniformly with respect to $q \in Q$, and nondecreasing in $s \ge 0$ } and $ \inf_{q \in Q} {D(0, q)} =\varpi
>0 $. (This means that there is some inherent, {density
unrelated, mortality.})
\end{itemize}
The {\em reproduction number} of strategy $q\in Q$ at population size $s$ is defined
by
\begin{equation}
\label{eq:rep-num}
\cR( s, q) = \frac{B(s,q)}{D(s,q)}.
\end{equation}
The {\em basic reproduction number} of strategy $q$ is defined by
\begin{equation}
\label{eq:rep-num-basic}
\cR_0(q) = \cR(0,q), \qquad q \in Q.
\end{equation}
The following additional assumption is made.
\begin{itemize}
\item[{\bf (A3)}] For each $q \in Q$ with $\cR_0(q) \ge 1$, there exists a unique
$K(q) \ge 0$ such that $\cR(K(q), q)$ $ =1$.

\item[] {If $\cR_0(q) < 1$, we define $K(q) =0$.}

\end{itemize}

{The number $K(q)$ in (A3) is the {\em carrying capacity} of
the environment if everyone in the population were
subject to strategy $q$.}
Since (A1)-(A3) imply that the function $K(\cdot)$ is continuous, it has a maximum and
a minimum on the compact set $Q$. We define
\begin{equation}
\label{Max}
K^\diamond = \max_{ q \in Q} K(q)
\end{equation}
and  \begin{equation} \label{Min} k_{\diamond}= \min_{ q \in Q} K(q).
\end{equation}

Let $Q^\di$ be the subset of $Q$ where the maximal carrying capacity
is taken,
\begin{equation}
\label{eq:Q-max}
Q^\di = \{q \in Q; K(q) = K^\di\}.
\end{equation}
Then $Q^\di$ is a nonempty compact subset of $Q$ and $K(q) = K^\di$
for all $q \in Q^\di$. Further, if $K^\di >0$,
\begin{equation}
\label{eq:carry-cap-more}
\left . \begin{array}{rl}\cR(K^\di, q) =1  &
 \\ \cR(x,q) >1, &  0 \le  x< K^\di
 \\
 \cR(x, q) < 1, & x > K^\di
 \end{array} \right \} \quad q \in Q^\di.
\end{equation}

\subsection{Main Theorem from \cite{CLEVACK}}
The following is the main well-posedness theorem taken from \cite{CLEVACK}.
\begin{theorem}
\label{main}
Assume that (A1)-(A2) hold. There exists a continuous dynamical system $({\cal M}_{w,+
}, C(Q, {\cal P}_w),\phi)$ where $\phi:
{\mathbb {R}}_+ \times{\cal M} _{w,+} \times C(Q, {\cal P}_w) \to {\cal M}_{w,+}
  $ satisfies the following:
\begin{enumerate}

\item The mapping $(t,u,\gamma) \mapsto \phi(t,u, \gamma)$ is  continuous.

\item  For fixed $ u, \gamma $, the mapping $t \mapsto \phi(t,
u,\gamma)$ is continuously differentiable in total variation, i.e.,
$\phi(\cdot, u, \gamma): {\mathbb {R}}_+ \to {\cal M}_{V,+}$ is continuously differentiable.
 \item For fixed $ u, \gamma $, the mapping $t \mapsto \phi(t, u,\gamma)$ is the unique \emph{solution} $\mu$ to
 \begin {equation}
 \left\{\begin{array}{rl}
 \label{M}
 \displaystyle {\mu}'(t)(E) = & \int_Q {B}(\bar \mu(t),  q) \,\gamma( q)(E)\,\mu(t)(d q)
\; - \; \displaystyle \int_E {D}(\bar \mu(t), q) \,\mu(t)(d q)
\\ = &  {F} (\mu(t), \gamma)(E),
 \\
\mu(0)= &u,
\end{array}\right.
\end{equation}
where $\bar \mu (t) = \mu(t) (Q)$.
\end{enumerate}
\end{theorem}


\section{ Asymptotic Results with an Arbitrary {Mutation Kernel} $\gamma$ }
\label{sec:asymp-gen}


In this section we begin studying the long time behavior with an arbitrary
{mutation} kernel $\gamma \in C(Q, {\cal P}_w)$. In particular, we provide sufficiency for permanence and uniform persistence. \\

\begin{lemma}
\label{re:pos-pres}
 Let (A1)-(A2) hold and let $\mu$ be a solution
 of (\ref{M}). Then the following holds.

\begin{itemize}
\item[(a)] If $E$ is a Borel subset of $Q$ and  $\mu(0)(E) > 0$, then $\mu(t) (E) > 0$ for all $t \ge 0$.

\end{itemize}

\noindent
Assume in addition that $B(x,q) > 0$ for all $x \ge 0$ and $q \in Q$. Then
the following holds.

\begin{itemize}
\item[(b)] If $E$ and $\tilde E$ are Borel subsets of $Q$ with $\inf_{q\in {\tilde E}} \gamma(q)(E) >0$ and
$\mu(t)({\tilde E}) >0$ for all $t >0$,
then $\mu(t)(E) > 0$ for all $t >0$.
\end{itemize}
\end{lemma}

\begin{proof}
{For a solution $\mu$ of (\ref{M}) and $\bar \mu(t) = \mu(t)(Q)$, set}
\begin{equation}
\label{eq:extreme-vitals}
b(t) = \min_{q\in Q} B(\bar \mu(t),q), \qquad \theta(t) = \max_{q \in Q} D(\bar \mu(t), q), \qquad t \ge 0.
\end{equation}
Since $B$ and $D$ are continuous and $Q$ compact, $b$ and
${\theta}$ are continuous. We have the differential inequalities,
\[
\mu'(t) (E) \ge  b(t) \int_Q \gamma(q)(E) \mu(t)(dq) - \theta(t) \mu(t) (E)
\ge - \theta(t) \mu(t)(E).
\]
The second inequality can be integrated as
\[
\mu(t) (E) \ge \mu(0)(E) \exp \Big ( - \int_0^t \theta(s) ds \Big),
\]
which provides the first statement. Let {$\tilde E$} also be a Borel subset of $Q$. Then
\[
\mu'(t) (E) \ge  b(t) \inf_{q \in {\tilde E}} \gamma(q)(E) \mu(t)({\tilde E}) - \theta(t) \mu(t) (E).
\]
This can be integrated as
$$
\mu(t)(E) \ge \inf_{q \in {\tilde E}} \gamma(q)(E) \int_0^t b(r) \mu(r)({\tilde E})
\exp \Big (- \int_r^t \theta(s)ds\Big ) dr.
$$
This implies the second statement.
\end{proof}

\subsection{Uniform eventual boundedness}
A system $ \frac{dx}{dt} =F(x)$ is called dissipative and its solution uniformly eventually bounded, if all solutions exist for all forward times and if there exists some $c>0$ such that $$ \limsup_{t \rightarrow \infty}||x(t)||< c $$ for all solutions $x$ \cite[pg. 153]{Thi03}.
Next, we will show that the solutions of (\ref{M}) are uniformly eventually bounded. Recall
that $\bar \mu(t)= \mu(t)(Q)$ denotes the total population size at time
$t$.

\begin{theorem}
\label{BS}(Bounds for Solution)
Assume that (A1)-(A3) hold. Then, for any solution $\mu$ of \eqref{M} and $\bar \mu
 = \mu( \cdot)(Q)$,
 we have the following:
 \begin{equation}\label{limbound}
\min\{ k_{\diamond}, \bar \mu(0)\} \le \bar \mu(t) \le \max
\{\bar \mu(0), K^{\diamond} \}, \qquad \text{for all }~  t
~\geq 0,
  \end{equation}
 and \begin{equation}\label{limsupbound}
k_{\diamond}  \leq \liminf_{t \to \infty} \bar \mu (t) \leq
\limsup_{t \to \infty} \bar \mu(t)  \le K^{\diamond} .
\end{equation}
Hence, if $k_{\diamond} >0$
 then the population is permanent.

\end{theorem}

\begin{remark} It seems at first glance that $k_{\diamond} >0$
 is too restrictive of an assumption
for proving persistence, i.e., $ \liminf_{t \rightarrow \infty}
\mu(t)(Q) >0 .$ However, if $k_{\diamond} =0$  and $ \gamma(\hat
q) = \delta_{\mathfrak{q}}$ (i.e., individuals with any strategy only reproduce individuals with strategy $\mathfrak{q}$), then it is an exercise to show that the
model converges to the zero measure even in setwise convergence.
\end{remark}

\begin{proof} We first prove the rightmost inequalities, i.e., those for $\max$ and $\limsup$.
 Let $\bar \mu (t)=\mu(t)(Q)$. First notice that
\begin{equation}
\label{TOTPOPBOUND}
 \begin{array}{lll}
 \bar \mu'(t) &=& \displaystyle
\int_{Q}\Bigl[B(\bar \mu(t),q)-D(\bar \mu(t),q)\Bigr]\mu(t)(d
q) \\[3mm]
&=& \displaystyle
 \int_{Q}\Bigl[\cR(\bar \mu(t),q )-1 \Bigr]D(\bar \mu (t),q)\mu(t)(d q).
\end{array}
\end{equation}

 For the second inequality in \eqref{limbound} we have cases.  First assume $K^{\diamond}
=0$. Using \eqref{TOTPOPBOUND} and the fact that for every $ q \in Q $,  $ \cR(\cdot,q )$
is nonincreasing  we see that $\bar \mu'(t) \leq 0$,
for all $t \geq 0 $. Now assume $0<K^{\diamond}  < \infty
$. Starting from \eqref{TOTPOPBOUND}, we see that if $\bar \mu(t)>
K^{\diamond} $ then $ \bar \mu' \leq 0$; therefore it follows from
basic analysis or \cite[Lemma A.6]{Thi03} that
\begin{equation}
\label{limbound1}
\bar \mu(t)\leqslant \max\{ K^{\diamond},
\bar \mu(0)\},
 \text { for all } t \ge~ 0 .
\end{equation}
In particular, $\bar \mu$ is bounded and we can define
\[
 \bar \mu^\infty := \limsup_{t \to \infty} \bar \mu(t).
\]
By the fluctuation lemma \cite{HHG}\cite[Prop.A.22]{Thi03}, there
exists a sequence $(t_j)$ such that $t_j \to \infty$, $\bar \mu(t_j) \to \bar \mu^\infty$
and $\bar\mu'(t_j) \to 0$ as $j \to \infty$. By (\ref{TOTPOPBOUND})
and the continuity of $\cR$ and $D$
\[
0 = \lim_{j \to \infty} \int_Q [ \cR(\bar \mu^\infty,q) -1 ]
D(\bar \mu^\infty, q ) \mu(t_j) (dq).
\]
Suppose that $\bar \mu^\infty > K^\di$. Then there exists some $\xi \in [0,1)$
and $\delta >0$
such that $\cR(\bar \mu^\infty, q) \le \xi$ and $D(\bar \mu^\infty, q ) \ge \delta$
for all $q \in Q$. So
\[
0 \le (\xi-1) \delta \bar \mu^\infty <0,
\]
a contradiction.

Exchanging $\liminf$  arguments for $\limsup$ arguments,
the lefthand inequalities are proved similarly.
\end{proof}

Recall $\phi$ from Theorem \ref{main}. Let $d (\nu, A) = \inf \{ p(\nu - u); u \in A\}$ be the distance from the
point $\nu \in \cM_{w+}$ to the set $A \subseteq \cM_{w+}$, where $p$ is the norm
defined in (\ref{eq:norm-weak}).

\begin{corollary}
\label{re:compact-glob-attr}
Assume that (A1)-(A3) hold. Then, for any $\gamma \in C(Q, \cP_w)$,
there exists a compact attractor of bounded sets, i.e., a  compact invariant subset $A_\gamma $ of $\cM_{w+}$ such that, for
all bounded subsets $B $ of $\cM_{w+}$, $ d( \phi(t, u, \gamma), A_\gamma) \to 0$
as $t \to \infty$ uniformly for $u \in B$.
Moreover, $\kappa (Q) \le K^\di$ for all $\kappa \in A_\gamma$,
and there is a compact set $ C$ such that $A_\gamma \subseteq
 C$ for all $\gamma \in C(Q, \cP_w)$.
Finally, the attractors $A_\gamma$ are upper semicontinuous, i.e.,
for any $ \eta \in C(Q, \cP_w)$,
\[
 \sup_{\nu \in A_\gamma} d (\nu, A_{\eta}) \to 0
\hbox{ as }\gamma \to \eta.
\]
\end{corollary}
\begin{proof} Consider the semiflow $\phi(\cdot, \gamma)$ defined in Theorem \ref{main}. By Theorems \ref{main}
and  \ref{BS}
and Lemma \ref{re:measures-weak-top},
in the language
of \cite[Def.2.25]{SmTh}, this semiflow is point-dissipative, asymptotically smooth,
and eventually bounded on bounded sets.
Existence of the attractors $A_\gamma$ now follows from \cite[Thm.2.33]{SmTh}.

We define a function $ V : \cM_{w+}(Q)\rightarrow \mathbb{R_+}$ by $V(\nu) = [\nu(Q) - K^\di]_+^2$,
where $r_+ = \max\{r,0\}$ is the positive part of a real number $r$.  $V$ is continuous by
Lemma \ref{re:semicont}. Furthermore,
 $r_+^2$ is differentiable and $\frac{d}{dr} r_+^2 = 2 r_+$.
{Then the orbital derivative of $V$ along (\ref{M}) \cite[p.313]{Hal} is
\[
\dot V(\nu) := \limsup_{t \to 0+} \frac{1}{t} ( V(\phi(t, \nu, \gamma)) - V(\nu))
=
2 [\nu(Q) - K^\di]_+ \int_Q [\cR(\nu(Q),q) -1 ] D(\nu(Q),q )
\nu (dq).
\]
This implies that $\dot V \le 0$ and $(d/dt) V (\phi(t,\nu, \gamma)) =
\dot V (\phi(t,\nu,\gamma)) \le 0$.
So  $V$ is a Lyapunov function in the sense of  Definition 2.49 \cite[ pg.52]{SmTh}.
Finally,
 $\{ \dot V = 0 \} = \{\nu ; \nu(Q) \le K^\di \} = \{V(\nu) =0\}$.
}{
Let $\kappa \in A_\gamma$. Since $A_\gamma$ is invariant, by \cite[Thm.1.40]{SmTh},
there exists a total solution $\mu : \R \to A_\gamma$
of (\ref{M}) such that $\mu (0) = \kappa$.
Then $(d/dt) V(\mu(t)) = \dot V (\mu(t)) \le 0$
and $V \circ \mu$ is decreasing on $\R$. Since $A_\gamma$ is compact,
$\mu$ has an $\alpha$-limit set $\alpha$ on which $V$ is constant. So $\dot V=0$
on  $\alpha$ which implies $\alpha \subseteq \{\nu; \nu(Q) \le K^\di\}= \{V(\nu) =0\}$.
Since $\alpha $ attracts $\mu (t)$ as $t \to - \infty$, $\lim_{t\to -\infty}
V(\mu(t)) =0$. Since $V(\mu(t))$ decreases, $0 = V(\mu(0)) = V(\kappa)$.
 This implies $\kappa (Q) \le K^\di$. This holds for any  $\kappa \in A_\gamma$
 and so $A_\gamma  \subseteq
\{\kappa \in \cM_{w,+}; \kappa (Q) \le K^\di\}$ which is a compact subset
in the $w^*$-topology by the Alaoglu-Bourbaki theorem.}

 The last statement follows from \cite[Prop.2.7]{Mag09}.
\end{proof}

We consider the state space $ X = \cM_{w,+} \times C(Q,\cP_w)$ with the product
topology and the semiflow $\Phi$
on $X$ given by $\Phi(t, (u, \gamma)) = (\phi(t, u, \gamma), \gamma)$.
$X$ can be made a metric space with the metric
$d_X ((u,\gamma), (v, \eta)) = p(u-v) + \|\gamma- \eta\|$ where $u,v \in \cM_{w\,+}$
and $\gamma, \eta \in C(Q, \cP_w)$.
 $\Phi$ is a continuous from $\R_+ \times X$ to $X$ by Theorem \ref{main}.

\begin{theorem}
\label{re:compact-attr-Cart}
 If $\Gamma$ is a compact subset of $C(Q, \cP_w)$, the restriction of the semiflow $\Phi$ to
  $X=\cM_{w,+} \times \Gamma$ has a compact attractor of bounded sets,
  $\cA_\Gamma$,
which satisfies $\cA_\Gamma = \bigcup_{\gamma \in \Gamma} (  A_\gamma \times \{\gamma\} )$,
where $A_\gamma \subseteq \cM_{w+}$ is the attractor of $\phi (\cdot, \gamma)$ in
Corollary \ref{re:compact-glob-attr}. In particular, all elements in $\cA_\Gamma$
are or the form $(\kappa, \gamma)$ with $\kappa (Q) \le K^\di$.
\end{theorem}

\begin{proof}
We apply \cite[Thm.2.33]{SmTh}.
By Corollary \ref{re:compact-glob-attr}, the compact set $\{u \in M_{w,+}; u(Q)\le K^\di\}
\times \Gamma$ attracts all  points in $X$. In particular, $\Phi$ is point-dissipative.

To check that $\Phi$ is asymptotically smooth, let $ B $ be a bounded subset
of $X$ that is forward invariant under $\Phi$. Then $\Phi(\R_+ \times B) \subseteq B$.
Since $X = \cM_{w,+} \times \Gamma$, $\Phi(\R_+ \times B) \subseteq K \times \Gamma$
with $K$ being a bounded subset of $\cM_{w,+}$. {By the Alaoglu-Bourbaki theorem,
 $\Phi(\R_+ \times B)$ is contained in a compact subset of $X$. This implies
 that $\Phi$ is asymptotically compact on $B$.}

$\Phi$ is bounded on every bounded subset $B$ of $X$ by Theorem \ref{BS}.

By \cite[Thm.2.33]{SmTh}, the restriction of $\Phi$ to $X$ has a compact attractor
of bounded sets, $\cA_\Gamma$.

Now let $\gamma \in \Gamma$ and $\tilde A_\gamma = \{\kappa \in \cM_{w+}; (\kappa, \gamma)
\in \cA_\Gamma \}$. Then $\tilde A_\gamma$ is a compact subset of $\cM_{w+}$
that is invariant under $\phi(\cdot, \gamma)$. Since $A_\gamma$ is the compact
attractor of bounded subsets for $\phi(\cdot, \gamma)$, $\tilde A_\gamma \subseteq A_\gamma$.

By Corollary \ref{re:compact-glob-attr}, $\tilde \cA_\Gamma=
\bigcup_{\gamma \in \Gamma} (  A_\gamma \times \{\gamma\})$ is a compact subset of $X$.
Since each $A_\gamma$ is invariant under $\phi(\cdot, \gamma)$, $\tilde \cA_\Gamma$
is invariant under $\Phi$. Since $\cA_\Gamma$ is the compact attractor of all bounded
subsets of $X$ for $\Phi$, $\tilde \cA_\Gamma\subseteq \cA_\Gamma$.
\end{proof}


\subsection{Persistence}


Next, we prove a persistence result for the case where
$k_\diamond$ is not necessarily positive and (A3) does not necessarily
hold. In order to obtain
population persistence, we impose a balancing inequality on some sets with strong strategies. A strong strategy, $q$, is one that has $\mathcal{R}(0,q)> 1.$ So if {$E \subseteq Q$} consists entirely of strong strategies, and if a member of $E$ contributes on average more than one of its offspring to $E$, then the population will persist. Mostly, the strong strategies need to play the balancing act \eqref{balance} if the population is to survive, and on average only extremely strong traits can afford to have large proportions of their offspring be weak.

\begin{definition}
\label{def:irred}
 Let $E$ be a Borel subset of $Q$. A kernel $\gamma \in C(Q, \cP_w )$
is called $E$-irreducible if for every solution $\mu$ of (\ref{M}) with
$\mu(0)(Q) > 0$ there exists some $r \ge 0$ such that $\mu(r)(E) >0$.

The kernel $\gamma$ is called uniformly $E$-irreducible if $r$ does not
depend on $\mu$.

\end{definition}

\begin{theorem}
\label{persistence}

Assume that (A1)-(A2) hold and let $\epsilon > 0$ and $ E \subseteq Q $ be a Borel set
such that
 \begin{equation}
\label{balance}
 \inf_{q \in E} \mathcal{R}(\epsilon,q) \gamma(q)(E) > 1 .
\end{equation}

\begin{itemize}
\item[(a)]
Then the population is uniformly weakly persistent in the
 sense that $\limsup_{t\to \infty} \mu(t)(Q) \ge \epsilon$ for all solutions
 with $\mu(0)(E) > 0$.

\item[(b)] Assume in addition that {the} kernel $\gamma$ is $E$-irreducible.
Then  $\limsup_{t\to \infty} \mu(t)(Q) \ge \epsilon$ for all solutions
 with $\mu(0)(Q) > 0$.
 \end{itemize}
\end{theorem}

\begin{proof}
(a) Assume that $\limsup_{t\to \infty} \mu(t)(Q)< \epsilon$. Then for
sufficiently large $t$,  since $B$ is nonincreasing and $D$ is
 nondecreasing  in the first variable and $\mu(t)(Q)$ is nonnegative,
$$
\begin{array}{ll}
 \mu' (t)(E) & \displaystyle \ge \int_E B(\epsilon,q) \gamma(q) (E) \mu(t)(d q) -
\int_E D(\epsilon,q) \mu(t)( dq)\\
& \displaystyle = \int_E [{\cal R} (\epsilon,q) \gamma(q)(E) -1] D(\epsilon,q)
\mu(t)(dq) . \end{array}
$$
Thus,
$$
\mu'(t)(E) \ge [\inf_{q\in E} {\cal R}(\epsilon,q) \gamma(q)(E) -1]
\int_E D(\epsilon,q) \mu(t)(dq).
$$
Then,
$$\mu'(t)(E) \ge [\inf_{q\in E} {\cal R}(\epsilon,q)\gamma(q)(E)-1]
\inf_{q\in E} D(\epsilon,q) \mu(t)(E).$$
So $\mu(t) (E) \to \infty$ because $\mu(0)(E) > 0$ and $\inf_{q\in E}
D(\epsilon,q)> 0$ by (A2).
This contradicts $\limsup_{t\to \infty} \mu(t) (Q) < \epsilon$.

(b) Now assume that $\gamma$ is $E$-irreducible and $\mu$ is a solution with
$\mu(0)(Q) > 0$. Then there exists some $r >0$ such that $\mu(r) (E) >0$
and $\limsup_{t \to \infty} \mu(t) (Q) > \epsilon$ by our previous result.
\end{proof}

\begin{theorem}
\label{re:persistence-strong}

Assume that (A1)-(A2)  hold and let  $ E \subseteq Q $ be a Borel set such that
$\gamma$ is $E$-irreducible and
 \begin{equation}\label{balance2} \inf_{q\in E}  \mathcal{R}(\epsilon,q) \gamma(q)(E) > 1, \text  {
for some  } \epsilon > 0 . \end{equation}
Then the following hold:
\begin{itemize} \item[(a)] The population is uniformly  persistent in the following sense:
There exist some $\epsilon_0 > 0$ such that $\liminf_{t\to \infty} \mu(t)(Q) \ge \epsilon_0$ for all solutions
with $\mu(0)(Q) > 0$.

\item[(b)] If (A3) holds as well and $\gamma$ is uniformly $E$-irreducible, then for every $f \in C_+(Q)$ with $\inf f(E) >0$
there exists some $\delta_f >0$ such that $\liminf_{t \to \infty}
\int_Q f(q) \mu(t)(dq) \ge \delta_f$ for all solutions with $\mu(0)(Q) > 0$.

\item[(c)] If (A3) holds and $E$ is open and $\gamma$ is uniformly $E$-irreducible, then there exists some $\delta_E >0$
such that $\liminf_{t \to \infty} \mu(t) (E) \ge \delta_E$
for all solutions with $\mu(0)(Q) > 0$.
\end{itemize}
\end{theorem}

\begin{proof}
\begin{itemize}
\item[(a)] By Theorem \ref{persistence}, in the language of \cite[A.5]{Thi03},
the semiflow induced by (\ref{M})
on $\cM_{w,+}$
is uniformly weakly $\rho$-persistent for $\rho: \cM_{w+} \to \R_+$,
$\rho(\mu) = \mu(Q)$. Further the sets $\{\rho \le c\}$ are compact
in $\cM_{w,+}$ for each $c >0$. By \cite[Thm.A.32]{Thi03},
the semiflow induced by (\ref{M})
on $\cM_{w,+}$ is uniformly $\rho$-persistent.
\item[(b)]Now assume (A3) in addition and that $\gamma$ is uniformly $E$-irreducible.
{For $f \in C_+(Q)$}, we apply \cite[Thm.4.21]{SmTh} with
\[
\tilde \rho (\mu) = \int_Q  f(q)  \mu(dq).
\]
By Theorem \ref{BS}, there exists a compact subset $C$ of $M_{w_+}$
such that $\mu(t) \in C$ for sufficiently large $t > 0$. Let $\mu(\cdot):
\R \to C_\epsilon $ be a total trajectory where $C_\epsilon = C \cap \{{\tilde \rho }\ge \epsilon\}$.
In our situation, a total trajectory is a solution $\mu$  that is defined for all $t \in \R$. Choose $r >0$ from the uniform $E$-irreducibility definition. Since $\mu(-r)(Q) > 0$,
we have $\mu(0)(E) >0$. Since  $\inf f(E)>0$, $\tilde \rho(\mu(0))> 0$.
~
Since $\tilde \rho$ is continuous, the semiflow is uniformly $\tilde \rho$-persistent
by \cite[Thm.4.21]{SmTh}.
~
\item[(c)] Now assume (A3) and that $E$ is open.
This time, we apply \cite[Thm4.21]{SmTh} with $\tilde \rho(\nu) = \nu (E)$. By Lemma \ref{re:semicont},
$\tilde \rho$ is lower semicontinuous.
\end{itemize}
\end{proof}

\begin{remark}
\label{re:gammas}
If $E \subseteq Q$ is open, the following hold:

\begin{itemize}
\item[(a)] Assumption  \eqref{balance2}  can be replaced by $\cR(0,q)\gamma(q)(E) >1$
for all $q \in \bar E$.

\item[(b)] If $\Gamma$ is a compact subset of $C(Q, \cP_w )$ and
$\cR(0,q)\gamma(q)(E) >1$
for all $q \in \bar E$ and all $\gamma \in \Gamma$, then there
exists some $\epsilon > 0$ such that $\inf_{q \in Q} \cR(\epsilon,q)\gamma(q)(E) >1$
for all $\gamma \in \Gamma$.
\end{itemize}
\end{remark}

\begin{proof} (a)
Suppose there is no $\epsilon >0$ such that
\[
\inf_{q \in E} \cR(\epsilon, q) \gamma(q,E) > 1.
\]

Let $(\epsilon_n)$ be a sequence of positive numbers such
that $\epsilon_n \to 0$. Then there exists a sequence $(q_n)$ in $E$
such that $\liminf_{n \to \infty} \cR(\epsilon_n, q_n) \gamma(q_n,E) \le 1$.
After choosing subsequences, we can assume that $q_n \to q$ for some $q \in \bar E$.
By continuity of $\cR$ and lower semi-continuity of $\gamma(\cdot) ( E)$
(recall Lemma \ref{re:semicont} (a)),
\[
\liminf_{n \to \infty} \cR(\epsilon_n, q_n) \gamma(q_n)(E)
\ge \liminf_{n \to \infty} \cR(\epsilon_n, q_n) \liminf_{n \to \infty} \gamma(q_n)(E)
\ge \cR(0,q) \gamma(q)(E) >1,
\]
a contradiction.

Part (b) is shown similarly using Lemma \ref{re:semicont} (b).
\end{proof}

\begin{theorem}
\label{re:persistence-unif}

Assume that (A1)-(A3)  hold and let  $ E \subseteq Q $ be an open set
and $\Gamma$ be a compact subset of $C(Q, \cP_w)$
with the
following properties:

\begin{enumerate}

\item   $\displaystyle    \mathcal{R}(0,q) \gamma(q)(E) > 1$
for all $q \in \bar E$ and all $\gamma \in \Gamma$.

\item All $\gamma \in \Gamma$ are $E$-irreducible.

\end{enumerate}

Then the population is uniformly  persistent in the following sense:
There exist some $\epsilon_0 > 0$ such that $\liminf_{t\to \infty} \mu(t)(E) \ge \epsilon_0$ for all solutions
$\mu$ of (\ref{M}) with $\mu(0)(Q) > 0$ and $\gamma \in \Gamma$.

\end{theorem}

\begin{proof}
We consider the state space $X = \cM_{w+} \times \Gamma$ with the product
topology and the semiflow $\Phi$
on $X$ given by $\Phi(t, (u, \gamma)) = (\phi(t, u, \gamma), \gamma)$.
$X$ can be made a metric space and $\Phi$ is continuous on $\R_+ \times X$.
We first choose the persistence function $\rho(u, \gamma) = u(Q)$.
By {Theorem \ref{persistence}} and Remark \ref{re:gammas} (b), $\Phi$ is uniformly
weakly $\rho$-persistent. By Theorem \ref{BS} and Lemma \ref{re:measures-weak-top}, there exists a compact
subset $C$ of $X$ such that, for all $(u,\gamma) \in X$,  $\Phi(t,u, \gamma) \in C$
for all sufficiently large $t > 0$. By \cite[Thm.4.13]{SmTh}, $\Phi$
is uniformly $\rho$-persistent. Now use a second persistence functions
$\tilde \rho(u, \gamma)= u(E)$.
By Lemma \ref{re:semicont}, $\tilde \rho$ is lower semicontinuous.
The statement now follows from \cite[Thm.4.21]{SmTh}
similarly as in the proof of Theorem \ref{re:persistence-strong}.
\end{proof}


\subsection{Robust persistence for optimum preserving mutation kernels}

We now consider $E = Q^\di$ and try to drop the irreducibility assumption.
As trade-off we assume that $Q^\di$ is open (which implies that $Q$ is disconnected)
and that there are no mutation losses for strategies in $Q^\di$.

\begin{definition}
\label{def:kernel-opt-pres}
A mutation kernel $\gamma^\di \in C(Q, \cP_w)$ is called  {\em optimum preserving}
if
 \[
 \gamma^\di (q) (Q^\di) =1 \hbox{  for all } q \in Q^\di.
 \]
\end{definition}

Assume that $Q^\di$ is an open subset of $Q$. Recall that $K^\di = \sup K(Q)$
and $Q^\di = \{q \in Q; K(q) = K^\di\}$. Since $\tilde Q = Q \setminus Q^\di$
is compact, $\tilde K := \sup K(\tilde Q) < K^\di$.

Choose some $\epsilon_0 > 0$ such that
\begin{equation}
\label{eq:pers-rob1}
\sup K(\tilde Q ) + 3  {\epsilon_0} < K^\di.
\end{equation}
Let $ E = Q^{\diamond} $ and $\psi( \gamma, q)=\gamma(q)(E)$  be as in Lemma \ref{re:semicont}. Then  $${ \cal F}(\gamma, q) =\cR( K^{\diamond} - \epsilon_0, q)\psi(\gamma,q) $$  is continuous, because $Q^{\diamond}$ is both open and closed. Hence,
there exists some $\delta_0 > 0$ such that
\begin{equation}
\label{eq:pers-rob2}
\inf_{q \in Q^\di} \cR(K^\di - {\epsilon_0}, q)  \gamma(q)(Q^\di) > 1, \qquad \|\gamma -\gamma^\di\|
< \delta_0.
\end{equation}

Here, $\|\cdot \|$ is the norm on $C(Q, \cM_p)$ defined in (\ref{eq:metric-mutation}). \\

\begin{theorem}
\label{re:pers-robust}
Assume (A1), (A2) and (A3), $K^\di > 0$, and that $Q^\di$ is an open subset of $Q$.
If  $\gamma^\di \in C(Q, \cP_w)$ is an optimum preserving  mutation kernel,
the following hold:

\begin{itemize}

\item[(A)]
Then there exists some $\tilde \delta \in (0, \delta_0) $ such that
$\displaystyle \liminf_{t \to \infty} \phi(t, u, \gamma)(Q^\di) \ge \tilde \delta $
for all $u \in \cM_{w+}$ with $u(Q^\di) > 0$ and all
$\gamma \in C(Q, \cP_w)$ with $\|\gamma- \gamma^\di\| < \tilde \delta$.

\item[(B)] For all $\gamma \in C(Q, \cP_w)$ with $\|\gamma- \gamma^\di\| < \tilde \delta$
with $\tilde \delta > 0$ from (A),
there exists a persistence attractor $\tilde A_\gamma$, i.e.,
a compact invariant stable subset $\tilde A_\gamma$ with the following properties:

\begin{itemize}

\item[$\bullet$] $\nu (Q^\di) > 0 $ for all $\nu \in \tilde A_\gamma $.

\item[$\bullet$] For all compact subsets $V$ of $\cM_{w+}$ with $\inf_{u \in V} u(Q^\di) > 0$
 there exists a neighborhood $U$ such that $d( \phi(t, u, \gamma), \tilde A_\gamma)
\to 0$ as $t \to \infty$ uniformly for $u \in U$.

\end{itemize}

\item[(C)] The attractors $\tilde A_\gamma$, $\|\gamma - \gamma^\di\|< \tilde \delta$, in part (B)
 are upper semicontinuous at $\gamma^\di$:
For all open subsets $W $ with $\tilde A_{\gamma^\di} \subseteq W \subseteq
\cM_{w+}$, there exists some $\delta_W \in (0,\tilde \delta)$ such that
$\tilde A_\gamma \subseteq W$ for all $\gamma \in C(Q, \cP_w)$ with $\|\gamma- \gamma^\di\|
< \delta_W$.

\end{itemize}

\end{theorem}

Here $d (\nu, A) = \inf \{ p(\nu - u); u \in A\}$ is the distance from the
point $\nu \in \cM_{w+}$ to the set $A \subseteq \cM_{w+}$ where $p$ is the norm
defined in (\ref{eq:norm-weak}).

Property (A) makes the semiflow $\Phi$  robustly persistent at
an optimum preserving mutation kernel $\gamma^\di$
as $\breve \epsilon >0$ can be chosen uniformly for all $\gamma$ in
a neighborhood of $\gamma^\di$
{(see \cite{HoSc, Sal} and the references therein).}

\begin{proof}
Suppose that (A) is false. Then there exist sequences
 $(\gamma_n)$ in $C(Q, \cP_w)$ and $(u_n) $ in $\cM_{w+}$ such that
 $\delta_0 \ge \|\gamma_n -\gamma^\di\| \to 0$,
$u_n(Q^\di) > 0$
and
\[
\liminf_{t \to \infty} \phi(t, u_n, \gamma_n)(Q^\di) \to 0, \qquad n \to \infty.
\]
Here $\delta_0 >0$ is from (\ref{eq:pers-rob2}).

 The set $\Gamma = \{\gamma_n; n \in \N\}
\cup \{\gamma^\di\}$ is compact in {$C(Q, \cP_w)$}.

Let $X = \cM_{w+} \times \Gamma$ with the metric  $\tilde D( (u, \gamma), (v, \eta))
= p (u-v) + \|\gamma- \eta\|$ for $u,v \in \cM_{w+}$ and $\gamma,\eta \in \Gamma$.
We apply \cite[Thm.8.20]{SmTh} to the semiflow $\Phi$ on $X$ given by
$\Phi(t,(u,\gamma)) = (\phi(t,u,\gamma), \gamma)$ and the persistence function
$\rho(u, \gamma) = u (Q^\di)$.

By Theorem \ref{re:compact-attr-Cart}, $\Phi$ has a compact attractor $\cA$ of bounded
sets on $\cM_{w+} \times \Gamma$.
Let
\begin{equation}
\label{eq:Xzero}
\begin{split}
X_0 := &  \{(u, \gamma) \in X; \forall t \ge 0: \rho (\Phi(t, (u, \gamma) ) =0  \}
\\= &
\{ (u, \gamma)\in X; \forall t \ge 0: \phi(t, u, \gamma)(Q^\di ) =0  \}.
\end{split}
\end{equation}
By \cite[Thm.5.21]{SmTh},  $\cA_0 = X_0 \cap \cA$ is a compact attractor of
compact sets in $X_0$ and attracts all subsets of $X_0$ that are attracted by
$\cA$. So $\cA_0$ attracts all bounded subsets of $X_0$ and  is isolated in $X_0$
by \cite[Thm.2.19]{SmTh}.
  By \cite[Thm.1.40]{SmTh}, $\cA_0$ is  acyclic in $X_0$
and, by \cite[Thm.2.17]{SmTh}, contains the $\omega$-limit sets of all points in $X_0$ under $\Phi$.

Next we show that $\cA_0$ is  uniformly weakly $\rho$-repelling for
$\rho(u,\gamma) = u(Q^\di)$.

Let $(u,\gamma) \in X_0$ and $\mu (t) = \phi(t, u, \gamma)$.
{Set $\tilde Q = Q \setminus Q^\di$.} Then
\[
\mu(t) (E) = \int_{\tilde Q} B\big ( \mu(t) (\tilde Q) , q \big ) \gamma(q)(E) \mu(t)(dq)
-
\int_E D \big (\mu(t)(\tilde Q),q \big ) \mu(t)(dq),
\qquad E \subseteq \tilde Q.
\]
So the restriction of $\Phi$ to $X_0$ corresponds to solutions
of (\ref{M}) with $Q$ being replaced by its compact subset $\tilde Q = Q \setminus
Q^\di$. By Theorem \ref{re:compact-attr-Cart},
\begin{equation}
\label{eq:Azero}
\cA_0 \subseteq \{(u, \gamma); \gamma \in \Gamma, u (\tilde Q) \le \tilde K\},
\qquad
\tilde K
:= \sup K(\tilde Q).
\end{equation}

Recall (\ref{eq:pers-rob1}) and (\ref{eq:pers-rob2}).

Suppose that $\cA_0$ is not uniformly weakly $\rho$-repelling.
If { $\epsilon_0>0$ is as in (\ref{eq:pers-rob1})}, then there exists some
$u \in \cM_{w+}$ and $\gamma \in \Gamma$ such that for $\mu = \phi(\cdot, u, \gamma)$
we have
$\mu(0)(Q^\di) > 0$ and $\limsup_{t \to \infty} d( \mu(t), \cA_0) < \epsilon_0.$
{By (\ref{eq:Xzero}) and (\ref{eq:Azero}),}
\[
\limsup_{t \to \infty} \mu(t)(Q^\di) < \epsilon_0
\quad \hbox{ and }
\quad \limsup_{t \to \infty} \mu(t)(\tilde Q) < \tilde K + \epsilon_0.
\]
By (\ref{eq:pers-rob1}) and these results for $\mu$,
for sufficiently large $t > 0$, $\mu(t) (Q) < K^\di - \epsilon$.
Also $\mu(t) (Q^\di) > 0$ for all $t \ge 0$. {By (\ref{eq:pers-rob2})}, there is some $r > 0$
such that for all $t \ge r >0$,
\[
\begin{split}
\mu'(t) (Q^\di ) \ge & \int_{Q^\di} B(K^\di - \epsilon, q) \gamma (q) (Q^\di)
\mu(t)(dq)
-
\int_{Q^\di} D (K^\di - \epsilon, q) \mu(t) (dq)
\\
= &
\int_{Q^\di} \big [\cR(K^\di - \epsilon, q)\gamma (q) (Q^\di)- 1\big] D(K^\di - \epsilon, q) \mu(t) (dq)
\ge
\breve \delta \mu(t) (Q^\di)
\end{split}
\]
with some $\breve \delta > 0$. So $\mu(t) (Q^\di ) \to \infty$, a contradiction.

This proves that $\cA_0$ is uniformly weakly $\rho$-repelling.
By \cite[Thm.8.20]{SmTh}, {with  $\Omega \subseteq M_1:=\cA_0 $},
$\Phi$ is uniformly weakly $\rho$-persistent on
$X$.
{\cite[Thm.4.13]{SmTh}
implies that $\Phi$ is uniformly $\rho$-persistent on $X$.
This contradicts
$\liminf_{t \to \infty} \phi(t, u_n, \gamma_n)(Q^\di) \to 0$ as $n \to \infty$
because $(u_n,\gamma_n) \in X$ and $\rho(u_n,\gamma_n) > 0$.\\
}

(B) follows from part (A) and \cite[Thm.5.6]{SmTh} applied to each semiflow $\phi(\cdot, \gamma)$
with $\|\gamma - \gamma^\di\| < \tilde \delta$.

(C) follows from \cite[Thm.1.1]{Mag09}.
\end{proof}



\section{Pure Selection Dynamics}
\label{sec:pure-sel}


 When attempting to analyze the asymptotic behavior of an EGT model, one
usually first forms the appropriate notion of an \emph{Evolutionary
Stable Strategy} or \emph{ESS.} The concept of an \emph{ESS} was
introduced into biology from the field of game theory by Maynard
Smith and Price to study the behavior of animal conflicts
\cite{Maynard2}. Intuitively an\emph{ ESS} is a strategy such that
if all members of a population adopt it, no differing behavior could
invade the population under the force of natural selection. So at
the ``equilibrium" of an \emph{ESS} all other strategies, if
present in small quantities, should have negative fitness and die
out.

We use the above discussion to define ESS and ASS as follows.
We define for two strategies $q$ and $\hat
q $ a relative fitness. Then using this definition we
define an \emph{ESS}. To this end define the relative fitness
between two strategies as
$$\lambda_R(q, \hat q) =\mathcal{R}(K(q), \hat q) -1.$$
This is clearly well-defined and is a measure of the long term
fitness of $\hat q$ when the subpopulation with trait $q$ is at its carrying capacity $K(q)$.

\begin{definition} A strategy $q$ is a  \textbf{(local) global \emph{ESS}} if $ q $
satisfies $$\lambda_{R}(q,\hat q) <
\lambda_{R}(q,q), \text{ for all }
\hat q,  \hat q \neq
q,  ~(\text{in a neighborhood of }
q).  $$ In this work an ESS is global unless explicitly mentioned as local.
\end{definition}

Notice that $\lambda_{R}(q,q) =0$
for all $ q $ and hence $q$ is an
(local) \emph{ESS} if and only if $\lambda_{R}(q, \hat
q) <0$ for all $\hat q \neq
q$ (in a neighborhood of $q)$. All
other strategies (in a neighborhood) have negative fitness when the subpopulation with trait $q$ is at its carrying capacity $K(q)$ and hence die out.

 \begin{remark} Assumption (A3) implies that $R(K(q),\hat q) \ne 1$ for any $\hat q \ne q$.
 Recall, that $I(q,\hat q)=R(K(q),\hat q)$ is the invasion reproductive number of strategy
 $\hat q$ with respect to strategy $q$, i.e., it is a measure of the ability of strategy
 $\hat q$ to invade strategy $q$ when the subpopulation with strategy $q$ is at its carrying capacity $K(q)$
 \cite{MM}.
 From the definition of an ESS it is evident that the following are equivalent:
 1) finding an ESS;
 2) finding a strategy $q$ such that the \textbf{relative fitness}
 $\lambda_R(q,\hat q) <0$ for all $\hat q \ne q$;
 3) finding a strategy $q$ such that the \textbf{invasion reproductive number}
 $I(q,\hat q) < 1$ for all  $\hat q \ne q$; 4)  finding a strategy $q$ that has
 the \textbf{largest carrying capacity} $K(q)$.
\end{remark}

The concept of an ESS is insufficient to determine the outcome of
the evolutionary game, since a strategy that is an \emph{ESS} need
not be an evolutionary attractor \cite[ch.6]{BrownVincent}. It is
not the case that all members of the population will end up playing
that strategy. An \emph{ESS} simply implies that if a population
adopts a certain strategy (phenotype, language or cultural norm
etc.), then no mutant small in quantity can invade or replace this
strategy.

\begin{definition} Suppose a population is evolving according to
\eqref{M}. If
$c_{q}\delta_{q}$ for some finite number
$c_{q}$  attracts any solution $\mu(t)$  of \eqref{M} satisfying $q \in supp(\mu(0))$,  then we call the strategy $q$ an
\textbf{Asymptotically Stable Strategy} or
\textbf{\emph{ASS}}. \end{definition}

This strategy, if it exists, is the endgame of the evolutionary
process. As it is attractive and once adopted, it cannot be invaded
or replaced.

Let $\gamma( q) =\delta_{q}  \in C(Q, {\cal P}_w)$ for all $q \in Q$ and $u\in
\mathcal{M}_+.$ Substituting these parameters in \eqref{M} and setting
$\mu(t) = \phi(t, u, \gamma)$ and $\bar \mu(t) = \mu(t)(Q)$,  one
obtains the pure selection model
\begin {equation}
\left\{\begin{array}{ll}
\label{selection}
 \displaystyle {\mu}^\prime (t)(E)
 =
 \int_E \left [B(\bar \mu(t),  q)-D(\bar \mu(t),q)
\right ]\mu(t)(d q) \\
\mu(0)=u.
\end{array}\right.
\end{equation}

{Using the uniqueness of solutions in Theorem \ref{main},
 one observes that the solution to \eqref{selection} satisfies the following integral
 representation :
\begin{equation}
\label{pureintrep}
\mu(t)(E)= \int_{E}\exp \Big (\int_{0}^{t}B(\bar \mu(\tau),q)-D(\bar \mu(\tau),q)d\tau \Big )\mu(0)(dq) .
\end{equation}}

The  following easy consequence will be used without further mentioning.

\begin{lemma}
\label{re:posit-pres}
 If $\mu(0)(E) =0$, then $\mu(t)(E) =0 $ for all $t \ge 0$.
If $\mu(0)(E) >0$, then $\mu(t)(E) >0 $ for all $t \ge 0$.
\end{lemma}

\begin{theorem}
 \label{re_equilibria-select}
Every $u \in \cM_{w+}$ with $u(Q^\di) = K^\di$ and $u (Q \setminus Q^\di) =0$ is an
equilibrium of (\ref{selection}). In particular, the collection $\cM^\di$
of all such measures is a compact invariant set.
\end{theorem}

\begin{proof} Let $u$ be as described above. Then $u(D) =0$ for all
Borel subsets
of $Q \setminus Q^\di$ and, for all Borel subsets $E$  of $Q$,
\[
\begin{split}
& \int_E [B(u(Q),q) - D(u(Q),q)] u(dq)
\\
= &
\int_{E \cap Q^\di} [B(K^\di,q) - D(K^\di,q)] u(dq)
+
\int_{E \setminus Q^\di} [B(u(Q),q) - D(u(Q),q)] u(dq)
\\
= &
0 +0 =0.
\end{split}
\]

Any set of equilibria is invariant. To show that $\cM^\di $ is closed, let
$(u_k)$ be a sequence in $\cM^\di$ and $u \in \cM_{w+}$ such that $u_k \to u$
in the weak$^*$ topology. Since $\chi_Q$ is continuous, $u(Q)= \lim_{j\to \infty} u_j(Q)
= K^\di$. Since $Q^\di$ is compact, $Q \setminus Q^\di$ is open and there exists
an increasing sequence $(f_j)$ of continuous functions such that $0 \le f_j \le \chi_{Q \setminus Q^\di}$ and $f_j \to \chi_{Q \setminus Q^\di}$ pointwise \cite[Thm.3.13]{ALI}. Then, for each $ j, k$,  $\int_Q f_j(q) u_k(dq)=0$ and so
$ \int_Q f_j(q) u(dq) =0$. Thus $ \int_Q \chi_{Q \setminus Q^{\diamond}}(q)u(dq) = u(Q \setminus Q^\di)=0$ by the monotone convergence
theorem.
\end{proof}

\begin{proposition}
\label{re:select1}
Assume (A1)-(A3) and $K^\di > 0$. For every $\epsilon \in (0, K^\di)$,
there exist an open set $U_\epsilon$ with $Q^\di \subseteq U_\epsilon \subseteq Q$
and some $\xi_\epsilon > 1$  such that
$\cR( K^\di- \epsilon, q) > \xi_\epsilon $ for all $ q \in U_\epsilon$.
Further,
if $\mu$ is a solution of (\ref{selection}) such that $\mu(0)(U_\epsilon )>0$,
then
\[
\limsup_{t \to \infty} \mu(t)( Q) \in ( K^\di - \epsilon, K^\di].
\]
\end{proposition}

\begin{proof}
We know from Theorem \ref{BS} that
\[
\bar\mu^\infty: = \limsup_{t\to \infty} \mu(t)(Q) \le K^\di
\]
for all solutions $\mu $  of (\ref{selection}).

Let $\epsilon \in (0, K^\di)$.
 For all $q \in Q^\di$,
\[
1 = \cR( K^\di, q) < \cR(K^\di -\epsilon, q ).
\]
Since $\cR( K^\di -\epsilon, \cdot)$ is continuous and $Q^\di$ is compact, there
exists some $\xi_\epsilon > 1$ such that $\cR( K^\di -\epsilon, q) > \xi_\epsilon$ for all $q \in Q^\di$.

We claim that there exists some open set $U_\epsilon$ with $Q^\di \subseteq U_\epsilon \subseteq Q  $ such that
\[
\cR( K^\di - \epsilon, q) > \xi_\epsilon ,
\qquad q \in U_\epsilon.
 \]
 Recall that $ U_\delta (Q^\di)=\{ \hat q\in Q; d(\hat q, Q^\di)< \delta \}$
 is an open subset of $Q$ that contains $Q^\di$ for any $\delta > 0$.
So, if our claim does not hold, there exists a sequence $(\hat q_n)$
in $Q$ such that $d(\hat q_n, Q^\di) \to 0$ and $\cR( K^\di -\epsilon, \hat q_n) \le \xi_\epsilon$.
By definition of the distance function, there exists a sequence $(q_n)$
in $Q^\di$ such that $d(q_n, \hat q_n) \to 0$. After choosing
subsequences, $q_n \to q$ for some $q \in Q^\di$ and also $\hat q_n \to q$.
By continuity, $\cR(K^\di - \epsilon, q) \le \xi_\epsilon$, a contradiction.

Now consider a solution $\mu$ of (\ref{selection}) with $\mu(0)(U_\epsilon) > 0$.
So ${\mu(t)( U_\epsilon) }> 0$
for all $t \ge 0$. Suppose  ${\bar\mu^\infty}<  K^\di -\epsilon$. Then there exists some
$r \ge 0$ such that $\bar \mu(t) < K^\di - \epsilon$ for all $t \ge r$. For $t \ge r$,
\[
\begin{split}
\mu'(t)(U_\epsilon) = & \int_{U_\epsilon} [\cR(\bar \mu(t), q) -1] D(\bar \mu(t),q) \mu (t)(dq)
\\
\ge &
\int_{U_\epsilon} [\cR( K^\di -\epsilon, q) -1] D(\bar \mu(t),q) \mu (t)(dq)
\\
\ge &
\int_{U_\epsilon} [\xi -1] D(\bar \mu(t),q) \mu (t)(dq)
\ge
[\xi -1] \inf_{q \in U_\epsilon} D(0,q) \mu (t)(U_\epsilon).
\end{split}
\]
Since $[\xi -1] \inf_{q \in U_\epsilon} D(0,q) >0$, ${\mu(t)(U_\epsilon)}$
increases exponentially and $\bar \mu(t) \ge \mu(t)(U_\epsilon)$ grows unbounded,
a contradiction.
\end{proof}

We add another assumption which states that the strategies that
maximize the carrying capacity are also { superior at all other relevant
population densities.}

\bigskip

\begin{itemize}
\item[{\bf (A4)}] Let $Q^\di$ be defined by (\ref{eq:Q-max}).
\begin{equation}
\hbox{ For each } q^\di \in Q^\di \hbox{ and }  q \in Q \setminus Q^\di, \quad
\cR(X, q^\di) > \cR(X, q) \hbox{ for all } X \in [k_\di, K^\di].
\end{equation}
\end{itemize}

\begin{proposition}
Let $K^\di >0$.
 Assume that $\cR(X, q) = L(X) M(q)$ for all $X \ge 0$
and $q \in Q$ with
continuous  functions $L: \R_+ \to (0,\infty)$ and $M: Q \to \R_+$.
Define $M^\circ = \max_{q \in Q} M(q)$ and $Q^\circ = \{q \in Q;
M(q) = M^\circ\}$. Then $Q^\circ = Q^\di$ and
(A4) follows.
\end{proposition}

This result is very similar to one for chemostats namely that
maximizing the basic reproduction number amounts to the same as minimizing
the break-even concentration if the reproduction number factorizes
as above \cite{SmTh13}.

\begin{proof}

Since $K^\di > 0$, $1 < \cR(0, q)= L(0) M(q) $ for some $q \in Q$.
This implies that $1 < L(0) M(q) $ for all $q \in Q^\circ$
and so $K(q) > 0$ for all $q \in Q^\circ$.

\begin{quote}
Step 1: $K(\cdot)$ is constant on $Q^\circ$
\end{quote}

Let $q_1, q_2 \in Q^\circ$.
Then $M(q_1) = M^\circ = M(q_2)$.
Recall that
\[
1= \cR(K(q_1), q_1) = L(K(q_1)) M(q_1)=
L(K(q_1)) M(q_2)= \cR(K(q_1),q_2).
\]
 Since, by assumption,  $K(q_2)$ is uniquely determined by
 $\cR(K(q_2), q_2)=1$, $K(q_1) = K(q_2)$.

\begin{quote}
Step 2: If $q^\circ \in Q^\circ$ and $q \in Q \setminus Q^\circ$,
then $K(q^\circ) > K(q)$.
\end{quote}

Suppose that $q^\circ   \in Q^\circ$ and $q \in Q \setminus Q^\circ$.
Then $K(q^\circ) >0$ and we can assume that $K(q) > 0$.

By definition of $K$,
\[
1= \cR(K(q^\circ), q^\circ) = L(K(q^\circ)) M(q^\circ) =
L(K(q^\circ)) M^\circ
\]
and
\[
1 = \cR(K(q), q) = L(K(q)) M(q) < L(K(q))M^\circ.
\]
Then $L(K(q))> L(K(q^\circ))$. Since $L$ is decreasing, $K(q)< K(q^\circ)$.

Step 1 and Step 2 imply that $Q^\circ= Q^\di$.
\end{proof}

The following example in which the birth rate is of Ricker type
and the death rate is constant shows that (A4) is very restrictive and that
without (A4) maximizing the carrying capacity may be different from
maximizing the basic reproduction number. We will learn in the next section
that, without (A4), it is the carrying capacity that is maximized.

\begin{example}
\[
B (x,q) = \kappa_q e^{- \eta_q x} , \qquad D(x,q) = e^{\theta x}.
\]
Then
\[
\cR(x,q) = \kappa_q e^{ -(\eta_q + \theta) x}.
\]
Further
\[
\cR_0(q) = \kappa_q, \qquad K(q) = \frac{\ln \kappa_q}{\eta_q + \theta}.
\]
Now let $Q = \{q_1, q_2, q_3\}.$
We choose $\kappa_{q_1} > \kappa_{q_2} >1\ge \kappa_{q_3}$, but
$\eta_{q_1} $ much larger than $\eta_{q_2} $. Then
strategy $q_1$ has a larger basic reproduction number but
a smaller carrying capacity than strategy $q_2$, $K^\di = K(q_2)$ and $k_\di = K(q_3)=0$.
So
\[
\cR(k_\di, q_2) = \cR(0, q_2) = \kappa_{q_2} < \kappa_{q_1} = \cR(k_\di, q_1),
\]
falsifying (A4).
It is not clear whether this counterexample works if $k_\di > 0$.
\end{example}

In this subsection we will sometimes assume the following:
\begin{itemize}

\item[(A5)] There is a unique strategy with largest carrying capacity,
 i.e.,
 there is a unique $q^\di$ such that $ K (q^\di) = K^{\diamond}.$

\end{itemize}
Under these assumptions, we show that if a population is evolving according to the pure
selection dynamics \eqref{selection}, then a multiple of $\delta_{q^\di}
$   attracts all solutions  $\mu(t)$ that  embrace $q^\di$ as
a possible strategy.

\begin{theorem}
\label{CSS1}  Assume that (A1)-(A5) hold, then $q^\di$ is an $ASS$.
That is, if the population $ \mu(t) $ is evolving according to the
pure selection dynamics \eqref{selection} and $ q^\di
 \in supp (\mu(0))$, then
$$\mu{(t)} \to K^\di \delta_{q^\di}, \quad t \to \infty.
$$ in the weak$^*$ topology.
\end{theorem}

Two technical propositions are required. We first show  that strategies different from the optimal  strategies are
not adopted in the long run.

\begin{proposition}
\label{extinction} Assume (A1)-(A4). Let $U_0$ be an open set such
that $Q^\di \subseteq U_0 \subseteq Q$. Then there exists some
open set $U$ with $Q^\di \subseteq U \subseteq Q$
such that  $\mu (t, Q \setminus U_0) \to 0$ as $t \to \infty$ for all solutions
$
\mu$ with $\mu(0)(U) > 0$.
\end{proposition}

\begin{proof} The following statement will provide the assertion of the proposition.

\begin{quote}
\begin{description}
\item[Claim:] If $\check q \in Q \setminus Q^\di$, then there exists some $\delta = \delta(\check q) >0$
such that ${\mu(t)( U_\delta(\check q))} \to 0$ as $ t \to \infty$ for all
solutions $\mu$ with $\mu(0)(U_\delta (Q^\di))>0$.
\end{description}

Here $U_\delta(\check q)$ denotes the $\delta$-neighborhood of $\check q$.
\end{quote}

We first show that this claim implies the assertion of the proposition, indeed.

Set $Q_0 = Q \setminus U_0$. Then $Q_0$ is a compact subset of $Q$ and $Q_0 \cap Q^\di
= \emptyset$. There is a finite subset $\check Q$ of $Q_0$ such that $Q_0$ is contained
in the union of finitely many open sets  $V_q = U_{\delta}(  q) $, $ q \in Q_0$,
where $\delta=\delta_q$ has been chosen according to the claim. Let $\epsilon = \min_{q \in Q_0} \delta_q$ and $U = U_\epsilon (Q^\di)$.
Then, for all solutions $\mu$ with $\mu(0)(U) > 0$,
\[
\mu(t)( Q_0) \le \sum_{q \in \check Q} \mu(t)( V_q ) \to 0, \qquad t \to \infty,
\]
because $\check Q$ is finite.

We now turn to proving the claim.

Let $\breve q \in Q \setminus Q^\di $ and $\delta > 0$.
Let

\begin{equation}
U_\delta = U_\delta (\breve q), \qquad V_\delta = U_\delta (Q^\di),
\end{equation}
where $U_\delta (Q^\di)= \{q \in Q; d(x, Q^\di)< \delta\}$ is the $\delta$-neighborhood of $Q^\di$ and $d$ the distance function extending  the metric $d$.

We consider
\begin{equation}
x(t)= \mu(t)( V_\delta), \qquad y(t) = \mu(t)( U_\delta), \qquad t \ge 0.
\end{equation}
Assume that  $x(0) > 0$, and
so $x(t) > 0$ for all $t > 0$. So we can also consider the function
\begin{equation}
z = y^\xi x^{-1},
\end{equation}
where the number $\xi>0$  will be suitably determined.
(A similar function has been considered in \cite{AA1,AMFH}.)

The strategy of our proof is to show that, for sufficiently small $\delta >0$,
$z(t) \to 0$ and so $y(t) \to 0$ because  $x$ is bounded.

We can assume that $y(0) > 0$, otherwise $y$ is identically 0.
Thus $y(t) >0$ for all $t \ge 0$. Notice that
\begin{equation}
\label{eq:quotient-function}
{\frac{z'}{z} = \xi \frac{y'}{y} -  \frac{x'}{x}.}
\end{equation}
Further
\[
y' = \int_{U_\delta} (B(\bar \mu,q) - D(\bar \mu,q)) \mu(dq),
\qquad
x' = \int_{V_\delta} (B(\bar \mu,q) - D(\bar \mu,q)) \mu(dq).
\]

By continuity, there exist $ q^\di \in \bar V_\delta$ and $\hat q \in \bar U_\delta$
such that $B(\bar \mu,q) - D(\bar \mu,q) \ge B(\bar \mu,q^\di ) - D(\bar \mu,q^\di)$ for all $q \in U_\delta$
and $B(\bar \mu,q) - D(\bar \mu,q) \le B(\bar \mu,\hat q ) - D(\bar \mu,\hat q)$ for all $q \in V_\delta$. { (Note $q^\di$ is not necessarily a point in $Q^{\di})$}.
Hence
\[
\begin{split}
y' \le & (B(\bar \mu,\hat q) - D(\bar \mu,\hat q )) y = D(\bar \mu, \hat q) (\cR(\bar \mu, \hat q) -1)y,
\\
x' \ge & (B(\bar \mu, q^\di) - D(\bar \mu, q^\di )) x = D(\bar \mu,  q^\di) (\cR(\bar \mu, q^\di ) -1)x.
\end{split}
\]
We substitute these formulas into (\ref{eq:quotient-function}),
\[
\frac{z'}{z} \le  \xi  D(\bar \mu, \hat q) (\cR(\bar \mu, \hat q) -1)-  D(\bar \mu,  q^\di) (\cR(\bar \mu,  q^\di) -1).
\]
This can be rewritten as
\[
\frac{z'}{z} \le [\xi  D(\bar \mu, \hat q)-  D(\bar \mu,  q^\di)] (\cR(\bar \mu, q^\di) -1)
+  D(\bar \mu, \hat q) [\cR(\bar \mu, \hat q) - \cR(\bar \mu,  q^\di) ].
\]
By continuity, compactness and (A4), one can find some $\eta > 0$ and some $\bar K > K^\di$ and, if $k_\di > 0$,
some $ \bar k \in [0, k_\di)$ such that, for all sufficiently small $\delta > 0$,
\[
\cR(\bar \mu, \hat q) - \cR(\bar \mu,q^\di) < - \eta
\]
whenever $\bar \mu \in [\bar k, \bar K]$ and $q^\di \in \bar V_\delta$
and $\hat q \in \bar U_\delta $.
By Theorem \ref{BS}, for sufficiently large $t > 0$, $\bar \mu \in [\bar k, \bar K]$ and
\[
\frac{z'}{z} \le [\xi  D(\bar \mu, \hat q)-  D(\bar \mu,  q^\di)] (\cR(\bar \mu, q^\di) -1)
- \eta  D(0, \hat q) .
\]
Recall that $D$ is bounded on $[\bar k, \bar K] \times Q$ and bounded away from
zero on $\R_+ \times Q$. So,
by choosing $\xi >0 $ small enough, we obtain that  $\xi  D(\bar \mu, \hat q)-  D(\bar \mu,  q^\di) \le 0$ for all sufficiently large $t$.
Since $\cR(\bar \mu, q) \ge 1$ for all $\bar \mu \in [0, K^\di]$ and $q \in Q^\di$,
for arbitrary $\epsilon > 0$ we can arrange by choosing $\delta >0$
small enough and $\bar K$ close enough to $K^\di$ that
\[
\cR(\bar \mu, q^\di) \ge 1 - \epsilon, \qquad \bar \mu \in [0, \bar K], q^\di \in \bar V_\delta.
\]
So
\[
\frac{z'}{z} \le [\xi  D(\bar \mu, \hat q)-  D(\bar \mu,  q^\di)] \epsilon
- \eta  D(0, \hat q) .
\]
By choosing $\delta >0$ small enough, we can have $\epsilon >0$
small enough to have the right hand side being smaller than a negative constant.
This implies that $z(t)$ decreases exponentially as we wanted to show.
\end{proof}


\begin{proposition}
\label{re:converge-total}
Assume (A1)-(A4). Let $K^\di >0$. Then, for any $\epsilon \in (0,K^\di)$,
there exists an open set $W_\epsilon$ such that  $Q^\di \subseteq W_\epsilon \subseteq Q$
and $\liminf_{t \to \infty} \mu(t)( Q) \ge K^\di -\epsilon$ for all
solutions of (\ref{selection}) with $\mu(0)(W_\epsilon) > 0$.
\end{proposition}

\begin{proof}
Let $\epsilon  \in (0,K^\di)$. Set $K_\epsilon = K^\di -\epsilon$.
By Proposition \ref{re:select1},  there exist a neighborhood $\tilde U_\epsilon$ of $Q^\di$ and some $\xi_\epsilon
 > 1$ such
that $\cR (s, q) \ge \xi_\epsilon$ for all $s\in [0,K_\epsilon]$ and $q \in \tilde U_\epsilon$.
By Proposition \ref{extinction}, there exists an open set $V_\epsilon$
such that $Q^\di \subseteq V_\epsilon \subseteq Q$ and
 $\mu(t)( Q \setminus \tilde U_\epsilon) \to 0$ as $t \to \infty$ for all solutions $\mu$
 with $\mu(0) (V_\epsilon) > 0$.
By Proposition \ref{re:select1}, $\limsup_{t\to \infty} \mu(t)( \tilde U_\epsilon) > K_\epsilon$
if $\mu(0) (\tilde U_\epsilon) > 0$. Set $W_\epsilon = \tilde U_\epsilon \cap {V_\epsilon}$.
Suppose
\[
\bar \mu_\infty:= \liminf_{t \to \infty} \mu(t)(Q) < K_\epsilon
\quad \hbox{ and }\mu(0)(W_\epsilon ) >0.
\]
Since ${\mu(t)(\tilde  U_\epsilon)}$ does not converge as $t \to \infty$, a version of the fluctuation lemma
\cite{HHG}\cite[Prop.A.20]{Thi03}
provides a sequence $(t_n)$ with $t_n \to \infty$ such that
\[
\mu(t_n)( \tilde U_\epsilon) \to
 \bar \mu_\infty < K_\epsilon
\]
and $(d/dt) \mu(t_n)( \tilde U_\epsilon) =0$.
Then
\[
0=  \mu'(t_n)( \tilde U_\epsilon) = \int_{\tilde U_\epsilon} [\cR(\bar \mu (t_n), q) -1] D(\bar \mu(t_n) ,q) \mu (t_n)( dq).
\]
For sufficiently large $n$, $\bar \mu(t_n) \le K^\di -\epsilon$ and
\[
0  \ge (\xi -1) \inf_{q \in Q} D(0,q) \mu(t_n)( \tilde U_\epsilon)  >0.
\]
This contradiction finishes the proof.
\end{proof}

The next result now follows in combination with Proposition \ref{re:select1}.

\begin{corollary}\label{corollary}
Assume (A1)-(A4). If $\mu$ is a solution of (\ref{selection}) such that
$\mu(0)(U) >0$ for all open sets $U$ with $Q^\di \subseteq U \subseteq Q$,
then $\mu(t)( Q) \to K^\di$ as $t \to \infty$.
\end{corollary}

The following result extends Theorem \ref{CSS1}, since Theorem \ref{CSS1} is an obvious corollary with $Q^\di =\{q^\di \}$.

\bigskip

\noindent
\begin{proposition}
\label{re:converge-prep}
 Assume (A1)-(A4). Let $\mu$ be a solution of (\ref{selection}) such
that $\mu(0)(U) > 0$ for all open sets $U$ with $Q^\di \subseteq U \subseteq Q$.
Then, for all $f \in C(Q)$,
\[
\liminf_{t\to \infty } \int_Q f(q) \mu(t)(dq) \ge K^\di \inf_{Q^\di} f
\]
and
\[
\limsup_{t\to \infty } \int_Q f(q) \mu(t)(dq) \le K^\di \sup_{Q^\di} f.
\]
\end{proposition}

\bigskip

\begin{proof}
Define
\[
f^\di = \sup_{Q^\di} f, \qquad f_\di = \inf_{Q^\di} f.
\]
Let $\epsilon > 0$ and $U = \{q \in Q; f(q) < f^\di + \epsilon\}$.
Since $f$ is continuous, $U$ is an open subset of $Q$ and $Q^\di \subseteq U \subseteq Q$.
By Proposition \ref{extinction}, ${\mu(t)( Q \setminus U)} \to 0$ as $t \to \infty$.
Now
\[
\begin{split}
& \limsup_{t \to \infty } \int_Q f(q) \mu(t)(dq)
\le
\limsup_{t \to \infty } \Big (\int_U f(q) \mu(t)(dq) + \sup_Q f \; \mu(t)(Q \setminus U) \Big )
\\
\le  &
(f^\di + \epsilon) \limsup_{t \to \infty } \mu(t)( U)
=
(f^\di + \epsilon) \limsup_{t \to \infty } \mu(t)( Q)
=
(f^\di + \epsilon) K^\di.
\end{split}
\]
Since this holds for any $\epsilon >0$, the statement
for the limit superior follows.
The proof for the limit inferior is similar.
\end{proof}

Recall the set of equilibrium measures
\[
\cM^\di = \{ \nu \in \cM_{w+}; \nu(Q) = \nu(Q^\di) = K^\di\}
\textcolor{red}{.}\]

\begin{proposition}
\label{re:converge-prep2}
 Assume (A1)-(A4). Let $\mu$ be a solution of (\ref{selection}) such
that $\mu(0)(U) > 0$ for all open sets $U$ with $Q^\di \subseteq U \subseteq Q$.

For  every sequence $(t_n)$ with $t_n \to \infty$,
there exists a subsequence $(t_{n_j})$ and some $\nu \in \cM^\di$ such that, for all $f \in {C(Q)}$,
\[
\int_Q f(q) \mu(t_{n_j})( dq) \to  \int_{Q^\di} f(q) \nu(dq), \qquad j \to \infty.
\]
\end{proposition}

\begin{proof}
Let $(t_n)$ be a sequence with $t_n \to \infty$.
By Lemma \ref{re:measures-weak-top}, there exists some $\nu \in \cM_+(Q)$
and a subsequence $(t_{n_j})$ such that
\[
\lim_{j \to \infty} \int_Q f(q) \mu(t_{n_j})(dq)
\to \int_Q f(q) \nu (dq), \qquad f \in {C(Q)}.
 \]

For $f \equiv 1$,  $\nu(Q) = \lim_{j \to \infty} \mu(t_{n_j}) (Q) =K^\di$ by Corollary \ref{corollary}. Let $U$ be open  such that $Q^\di \subseteq U \subseteq Q$.
Then there exists some $\delta > 0$ such that $U \supseteq \bar V_\delta$
where $\bar V_\delta$ is the closure of the open set  $V_\delta:=
\{q \in Q ; d(q, Q^\di) < \delta\}$. Now $Q \setminus U$ is compact and
$Q \setminus \bar V_\delta$ is open  and $Q \setminus U \subseteq Q \setminus \bar V_\delta$.
So there exists some $f \in {C_{+}(Q)}$ with values between 0 and 1 such that $f \equiv 1 $ on $Q \setminus U$
and $f \equiv 0$ on $\bar V_\delta$.

By Proposition \ref{extinction},
\[
\nu (Q \setminus U) \le \int_Q f(q) \nu(dq)
=
\lim_{j\to \infty} \int_Q f(q)  \mu(t_{n_j})( dq)
\le
\limsup_{j\to \infty} { \mu(t_{n_j})(Q \setminus V_\delta)} =0.
\]
Now $Q^\di = \bigcap_{m\in \N}  V_{1/m} $ and
$Q \setminus Q^\di = \bigcup_{m\in \N} (Q \setminus  V_{1/m} )$.
Since $\nu$ is countably additive,
$
\nu (Q \setminus Q^\di) = \lim_{m \to \infty } \nu (Q \setminus  V_{1/m} )=0.
$

\end{proof}

The following result also extends Theorem \ref{CSS1}.

\begin{theorem}
\label{re:meas-select}
Assume (A1)-(A4).   Let $\mu$ be a solution of (\ref{selection}) such
that $\mu(0)(U) > 0$ for all open sets $U$ with $Q^\di \subseteq U \subseteq Q$.
Then there exists a function $\mu^\di: \R_+ \to \cM^\di$
with the following properties.

\begin{itemize}

\item[(a)] For all $B \in \cB$, $\mu^\di (\cdot)( B)$ is measurable on $\R_+$;

\item[(b)]
$ \displaystyle \int_Q f(q) \mu(t)(dq)  - \int_{Q^\di} f(q) \mu^\di (t)(dq) \to 0,
\qquad t \to \infty,
$
for all $f \in C(Q)$.

\end{itemize}
\end{theorem}

\begin{proof}
$\cM^\di$ is compact and sequentially compact with respect to
the weak$^*$ topology and the weak$^*$ topology on $\cM^\di$ can be induced by the norm
$p$
in Lemma \ref{re:measures-weak-top}.
Now let $\mu(\cdot) $ be a solution of the pure selection equation. Define
$g: \R_+ \times \cM^\di\to \R_+$ by
\[
g(t, \nu) = p(\mu(t) - \nu), \qquad t \ge 0, \nu \in \cM^\di.
\]
Then $g$ is continuous on $\R_+ \times \cM^\di$.  Since  $\cM^\di$ is a compact metric space,
it is complete and separable.
By a measurable selection theorem
(see \cite{BrPu}, e.g.), there exists a Borel measurable function
$\mu^\di: \R_+ \to \cM^\di$ such that
\begin{equation}
\label{eq:meas-select}
g(t, \mu^\di(t))= \inf_{\nu \in \cM^\di} g(t, \nu).
\end{equation}
We claim that $g(t, \mu^\di(t)) \to 0$ as $t \to \infty$. Suppose not.
Then there exists some $\epsilon > 0$ and a sequence $(t_n)$ with $t_n \to \infty$
such such $g(t_n, \mu^\di(t_n)) \ge \epsilon$ for all $n \in \N$.
By Proposition \ref{re:converge-prep2}, there exists some $\nu \in \cM^\di $ and a subsequence
$(t_{n_j})$ such that $g(t_{n_j}, \nu)= p(\mu (t_{n_j})- \nu) \to 0$.
By
(\ref{eq:meas-select}), $g(t_{n_j}, \nu (t_{n_j})) \to 0$, a contradiction.
By construction, $p(\mu(t)- \mu^\di(t)) \to 0$ as $t \to \infty$ and so
\[
\int_Q f(q) \mu(t)(dq) -  \int_{Q^\di} f(q) \mu^\di(t)(dq) \to 0, \qquad t \to \infty,
f \in C(Q).
\]
For all $f \in C(Q)$, $\int_Q f(q) \mu^\di(t)(dq)$ is a Borel measurable function
of $t$. Standard arguments imply that $\mu^\di (t)( B)$ is a Borel measurable function
of $t$ first for all open subsets of $Q$ and then for all Borel subsets of $Q$.
\end{proof}


\section{Directed mutation kernels}
\label{sec:dir-mut}

Alternatively to (A4), we  assume that $Q^\di$ is an open subset of $Q$.
We make this assumption
because it makes $\chi_{Q^\di}$ continuous.
It has the unfortunate consequence that $Q^\di$ is separated
from the rest of $Q$ and that $Q$ is not connected.

\begin{definition}
\label{def:directed}
Let $q^\di \in Q^\di$. A mutation kernel $\gamma: Q \to \cP_{w}$ is called  {\em directed
to $q^\di $} if $q^\di$ is an isolated point of $Q$
and the following hold:

\begin{itemize}

\item[(a)] For all $q \in Q^\di $,  $ \gamma(q)( \{q^\di\}) > 0$
and $\gamma (q) ( Q^\di) =1$.

\item[(d)] $\gamma (q^\di) (\{q^\di\}) =1$.

\end{itemize}
\end{definition}

It is easy to see that there is at most one $q^\di \in Q^\di$ a
mutation kernel can be directed to. Notice that every directed mutation
kernel is optimum preserving (Definition \ref{def:kernel-opt-pres}). In turn, if $Q^\di =\{q^\di\}$ and $\gamma$
is optimum preserving, then $\gamma$ is directed to $q^\di$. If $\gamma$
is a  mutation kernel directed to $q^\di$, then $K^\di \delta_{q^\di}$
is an equilibrium of (\ref{M1}). If $\gamma$ is the no-mutation kernel
$\gamma(q) = \delta_q$ and $Q^\di =\{q^\di\}$ is a singleton set and $q^\di$
is an isolated point of $Q$, then
the no-mutation kernel is trivially directed to $q^\di$.

\begin{theorem}
\label{re:directed}
Assume (A1)-(A3).
Let $q^\di \in Q^\di $ and $\gamma$ be a mutation kernel directed to $q^\di $.
Then, for all compact subsets $C$ of $\cM_{w+}$ with $u (Q^\di) > 0$
for all $u \in C$, $\phi (t, u, \gamma) \to K^\di  \delta_{q^\di}$
as $t \to \infty$
uniformly for $u \in C$.
\end{theorem}

\begin{proof}
By Theorem \ref{re:pers-robust} (c), in the language of \cite{SmTh},
the semiflow $\phi (\cdot, \gamma)$ induced by (\ref{M1})
is uniformly $\rho$-persistent for $\rho (\nu) = \nu (Q^\di)$.
 By Corollary \ref{re:compact-glob-attr}, this semiflow has a compact attractor $A_{\gamma}$
of bounded sets in $\cM_{w+}$ with $\nu (Q) \le K^\di$ for all $\nu \in A_{\gamma}$.
By Theorem \ref{re:pers-robust} (B), the semiflow has a $\rho$-persistence
attractor $A_1\subseteq A$, i.e.,
a compact invariant set $A_1\subseteq A$ which attracts all compact sets $V$
in $\cM_{w_+}$ for which $\nu(Q^\di) >0$ for all $\nu \in V$.

We claim that $A_1 = \{K^\di \delta_{q^\di}\}$.
We apply \cite[Thm.2.53]{SmTh} with $A= A_1$ and $\tilde A = \{K^\di \delta_{q^\di}\}$.

Let  $\mu: \R \to A_1$ be a  solution of (\ref{M1}) on $\R$.
Since $\mu$ is defined on the whole real line, it corresponds to a total trajectory.

Since $\mu$ takes its values in the $\rho$-persistence attractor, we have
$\inf_{t \in \R} \mu(t)(Q^\di) > 0$.
By a similar proof as the one of Lemma \ref{re:pos-pres} (b),
$\mu(t)(\{q^\di\})>0$
for all $t \in \R$ because of Definition \ref{def:directed} (a).

 For $\nu \in A_1$ with $\nu (\{q^\di\}) > 0$,
we define the Volterra type Lyapunov-function-to-be
\begin{equation}
L(\nu) = \nu(\{q^\di\}) + K^\di (\ln K^\di - \ln  \nu(\{q^\di\})
+
c \nu (Q \setminus Q^\di).
\end{equation}

It follows from our assumptions that $\chi_{\{q^\di\}}$ and $\chi_{Q \setminus
Q^\di}$
are continuous. So $L$ depends continuously on $\nu$ in the $w^*$ topology.
$L(\mu(t))$ is differentiable in $t$ and $(d/dt) L(\mu(t)) = \dot L (\mu(t))$ where $\dot L$ is
the orbital derivative of $L$ along (\ref{M}),
\[
\begin{split}
\dot L(\nu) = & \int_Q B(\nu (Q), q) \gamma(q)( \{q^\di\}) \nu(dq)
\Big (1 - \frac{K^\di}{\nu(\{q^\di\})}\Big )
\\
- &
D(\nu(Q), q^\di) (\nu(\{q^\di\}) - K^\di )
\\
+ &
c \int_{Q } B(\nu(Q),q ) \gamma(q)( Q\setminus Q^\di) \nu(dq)
\\
- &
c \int_{Q\setminus Q^\di} D(\nu(Q),q ) \nu(dq).
\end{split}
\]
 Set $G (x,q) = B(x,q) - D(x,q)$ for $x \ge 0$. Since $ \gamma(q) (Q \setminus Q^\di) =0
 $ for all $q \in Q^\di$ and $ \gamma (q^\di) ( \{q^\di\}) =1 $ by assumption,
\[
\begin{split}
\dot L(\nu) \le & \int_{Q \setminus \{q^\di\}} B(\nu (Q), q) \gamma(q)( \{q^\di\}) \nu(dq)
\Big (1 - \frac{K^\di}{\nu(\{q^\di\})}\Big )
\\
+ &
G(\nu(Q), q^\di) (\nu(\{q^\di\}) - K^\di )
\\
+ &
c \int_{Q \setminus Q^\di} G(\nu(Q),q )  \nu(dq).
\end{split}
\]
Since $\nu$ is an element in the global attractor, $\nu(\{q^\di\}) \le K^\di$
and the
first term on the right hand side is nonpositive,
\[
\begin{split}
\dot L(\nu) \le & \int_{Q^\di \setminus \{q^\di\}} B(\nu (Q), q) \gamma(q) ( \{q^\di\}) \nu(dq)
\Big (1 - \frac{K^\di}{\nu(\{q^\di\})}\Big )
\\
 & +
G(\nu(Q), q^\di) (\nu(\{q^\di\}) - K^\di ) +
c \int_{Q\setminus Q^\di} G(\nu(Q),q ) \nu(dq).
\\
\end{split}
\]
Define
\[
b = \inf \{ B(x,q); 0 \le x \le K^\di, q \in Q^\di\},
\qquad
\gamma_\di =\inf_{q \in Q^\di\setminus \{q^\di\}} \gamma (q)( \{q^\di\}).
\]
It follows from (A3) that $b > 0$. Since $\gamma^\di$ is directed towards
$q^\di$, $\gamma_\di > 0$.
After rearrangement,
\[
\begin{split}
\dot L(\nu) \le  &
\; b \gamma_\di \nu(Q^\di \setminus \{q^\di\})
\Big (1 - \frac{K^\di}{\nu (\{q^\di\})} \Big )
\\
& +
G(\nu(Q), q^\di) (\nu(Q) - K^\di )
+
G(\nu(Q), q^\di) (\nu(\{q^\di\}) - \nu(Q) )
\\
 & +
c \int_{Q \setminus Q^\di} G(\nu(Q),q )  \nu(dq).
\end{split}
\]
Since $\nu(Q) \le K^\di$, $G(\nu(Q), q^\di) \ge 0$ and so
\[
\begin{split}
\dot L(\nu) \le  &
\; b \gamma_\di \nu(Q^\di \setminus \{q^\di\})
\Big (1 - \frac{K^\di}{\nu (\{q^\di\})} \Big )
+
G(\nu(Q), q^\di) (\nu(Q) - K^\di )
\\&+
G(\nu(Q), q^\di) (\nu({Q^\di}) - \nu(Q) )
+
c \int_{Q \setminus Q^\di} G(\nu(Q),q )  \nu(dq).
\end{split}
\]
We rearrange,
\[
\begin{split}
\dot L(\nu) \le  &
\;b \gamma_\di \nu(Q^\di \setminus \{q^\di\})
\Big (1 - \frac{K^\di}{\nu (\{q^\di\})} \Big )
+
G(\nu(Q), q^\di) (\nu(Q) - K^\di )
\\&
+
 \int_{Q \setminus Q^\di} [cG(\nu(Q),q ) - G(\nu(Q), q^\di)]  \nu(dq).
\end{split}
\]
 Let $\tilde K = \sup_{q \in Q \setminus Q^\di} K(q)$. Since $Q \setminus Q^\di$ is  compact,
 $\tilde K < K^\di$. Let $q \in Q \setminus Q^\di$. Notice that
$G(X,  q^\di) > 0$ for all  $X \in [0,\tilde K]$. So, by choosing $c>0$
small enough, we can achieve that
\[
 cG(x, q) - G(x,  q^\di) <0 , \qquad x \in [0,\tilde K], q \in Q \setminus Q^\di.
 \]
For $x \in (\tilde K, K^\di]$, we have $G(x,q)  <0$ for $q \in Q \setminus Q^\di$,
and for $x \in [\tilde K, K^\di)$
we have $G(x,  q^\di)>0$. So for all
$q \in Q \setminus Q^\di$ and $x \in [\tilde K, K^\di]$,
\[
 cG(x, q) -  G(x,  q^\di) <0 .
 \]
In combination, this inequality holds for all $q \in Q \setminus Q^\di$ and $x \in [0,K^\di]$.

So, if $c>0$ is chosen
 small enough, $\dot L(\nu) \le 0$
and $\dot L(\nu) =0$  only if $\nu(Q) =K^\di $ and $\nu(Q\setminus Q^\di)=0 $
and so  $\nu(Q^\di)=K$ as well.
Moreover $\dot L(\nu) =0$ only if $\nu(\{q^\di\} ) = K^\di$
or $\nu (Q \setminus \{q^\di\}) =0$. Combined with the other
information, $\dot L(\nu) =0$ only if $\nu(Q) =K^\di = \nu(\{q^\di\})$,
i.e., $\nu $ is the point measure concentrated at $q^\di$  taking the value $K^\di$.

Recall that we apply \cite[Thm.2.53]{SmTh} with $A= A_1$ and $\tilde A = \{K^\di \delta_{q^\di}\}$.
All we need to show is that every solution $\mu : \R
\to A_1 $ of (\ref{M}) with $\dot L (\mu(t)) = 0 $ satisfies
$\mu(t) = K^\di \delta_{q^\di}$ which we just did.
\end{proof}


\begin{theorem}
Assume (A1)-(A3) and let $Q^\di$ be open in $Q$ and
$q^\di \in Q^\di$ be an isolated point of $Q$
and $\gamma^\di $ be a mutation
kernel directed towards $q^\di$. Assume that $K^\di > 0$.

Then, for any $\epsilon > 0$,  there exists some $\delta_\epsilon > 0$ such that
\[
\limsup_{t\to \infty} \phi(t, u, \gamma)  (Q \setminus \{q^\di\}) < \epsilon,
\quad
\limsup_{t\to \infty} |\phi(t,u, \gamma ) (Q) - K^\di| < \epsilon
\]
for all $\gamma \in C(Q, \cP_w)$ with $\|\gamma- \gamma^\di\| < \delta_\epsilon$
and all $u \in \cM_{w+}$ with $u(Q^\di)>0$.
\end{theorem}

\begin{proof}
Let $\epsilon > 0$ and let $W$ be the set of measures $\nu$ with $\nu(Q\setminus \{q^\di\})
< \epsilon $ and $|\nu (Q) - K|< \epsilon$. $W$ is an open neighborhood of $
K^\di \delta_{q^\di}$
in the weak$^*$-topology.
By Theorem \ref{re:pers-robust}, there exists some $\delta_\epsilon > 0$
such that $\tilde A_\gamma \subseteq W$ for all $\gamma \in C(Q, \cP_w) $
with $\|\gamma- \gamma^\di\| < \delta_\epsilon $. Since $\tilde A_\gamma$
is the persistence attractor for $\phi(\cdot, \gamma)$, for any $u \in \cP_{w+}$
with $u(Q^\di) > 0$ there exists some $t_u > 0$ such that $\phi (t,u,\gamma) \in W$
for $t \ge t_u$. This implies the assertion.
\end{proof}


\section{Discrete strategies and small mutations}
\label{sec:disc}

The results of this section are for finitely many strategies.  This means
that $Q$ is a finite set. Any finite set becomes a metric space if equipped
with the discrete metric to which any other metric on it is equivalent.

Notice that all points of $Q$ are isolated and all subsets of $Q$ are both
compact and open.

The main result of this section is to
demonstrate that, if there is
a unique strategy $q^\di \in Q$ under which the carrying capacity is maximal, $K(q^\di)
= K^\di$, then
there is a neighborhood around the pure selection
kernel where unique  equilibria are obtained which attract
all solutions which adopt this strategy at least partially.

We first show that in the above case the model \eqref{M}
reduces to a system of ordinary differential equations. To this end,
let $Q=\{q_{i}\}_{i=1}^{N}$.
Then $\mathcal{M}= \text{span}_{\mathbb{R}}\{\delta_{e_{i}}\}$
can be identified with $\R^N$ and
$C(Q,{\cal P}_w )$ with the set $\Gamma$ of nonnegative
$ N\times N$ matrices whose rows sum to
one.
  Let $ x_j(t) =\mu(t)(\{q_j\})$, $B_{j}(\bar x) = B(\bar x, q_j)$ and $D_{j}(\bar x) = D(\bar x, q_j)$
  where $ \bar x=\sum_{j=1}^{N} x_j$, then the  system \eqref{M} reduces to the following
differential equations system:

\begin {equation}
\left\{\begin{array}{ll}
\displaystyle  \frac{d}{dt}{x_j}(t;u, \gamma)
=
\sum_{i=1}^N  B_i (\bar x(t))x_i(t)\gamma_{ij} - D_j(\bar x(t))x_j(t), \qquad  j = 1, . . . , N ,
\\
x_j(0;u,\gamma) = u_j,
\end{array}\right.
\label{discrete}
\end{equation}
where $\gamma=\{ \gamma_{ij} \}$ is a row stochastic matrix.
The pure selection kernel $ \gamma( \hat q) = \delta_{\hat q} \in C(Q,{\cal P}_w )$
is represented as the $\I = I_{N\times N}$ identity matrix.

Note that in \eqref{discrete} $\gamma_{ij}$ represents the proportion of strategy $i$
 offspring that belong to strategy $j$. Since the sum of the proportions of offspring of strategy $i$ must be one, $\sum_{j=1}^N \gamma_{ij}= 1$
 for all $i=1, \ldots, N$. The norm on $C(Q, \cP_w)$ given by (\ref{eq:metric-mutation})
is equivalent to any of the matrix norms on the set  $\Gamma$ of row stochastic matrices.
Notice that $\Gamma$ is compact.

 Furthermore, let $K_j$ denote $K(q_j)$, the carrying capacity of strategy $j$.

The goal in this section is to study the dynamics of \eqref{discrete} when $\gamma$ is a small perturbation of the identity. To this end, we assume for the rest of this section that the fittest strategy is unique and that it
occurs at $q_1$, without loss of generality.
Thus, in this case $K^\diamond = K_j = K_1$. We assume $K_1 > 0$ and $K_1 > K_j$
for $j=2, \ldots, N$ and

\begin{itemize}
\item [{\bf (A6)}] $B_j$ and $D_j$ are continuously differentiable
{on $[0,\infty)$}  and
$ \displaystyle B_1'(K_1) - D_1'(K_1) < 0 $.
\end{itemize}

To  establish our asymptotic behavior result we recall the following theorem:
{Let
$x:[0,\infty) \to \R^N$ and}
\begin{equation}
\label{finite}
{x}^\prime = f(x, \lambda),
\end{equation}
where $ f: U \times \Lambda \rightarrow \mathbb{R}^N $ is
continuous, $ U \subseteq \mathbb{R}^{N}$, $\Lambda \subseteq
\mathbb{R}^{k}$ and $ \frac{\partial f}{\partial x} (x,\lambda) $ is continuous on $ U
\times \Lambda $. We write $x(t, z, \lambda) $ for the solution of
\eqref{finite} satisfying $x(0)=z.$

\begin{theorem} \label{pert1} (Smith and Waltman \cite{SmithWalt})
Assume that $(x_0, \lambda_{0}) \in U\times\Lambda $, $x_0 \in
int(U)$, $f(x_0, \lambda_0) =0$, all eigenvalues of $\frac{\partial f}{\partial x}(x_0,
\lambda_0)$  have negative real part, and $x_0$ is globally
attracting in $U$ for solutions of \eqref{finite} with $\lambda
=\lambda_0 $. If
\begin{description}
\item (H1) there exists a compact set $ \mathcal{D} \subseteq U$  such that
for each $ \lambda \in \Lambda$ and each $ z \in U$, $x(t,z,
\lambda) \in \mathcal{D}$ for all large t,
\end{description}
then there exists $\epsilon >0$ and a unique point $\widehat{x}(
\lambda) \in U$ for $\lambda \in B_{\Lambda}( \lambda_0, \epsilon)$ (the ball
in $\Lambda$ of radius $\epsilon$ around $\lambda_0$)
such that

\begin{itemize}
\item
$\hat x(\lambda)$ depends continuously on $\lambda \in B_{\Lambda}( \lambda_0, \epsilon)$ and $\hat x (\lambda_0) = x_0$,

\item $f(\widehat{x}( \lambda),\lambda)=0$ and

\item $x(t,z, \lambda)
\rightarrow \widehat{x}(\lambda)$, as $t\rightarrow \infty$ for all
$z \in U$ and $\lambda \in B_{\Lambda}( \lambda_0, \epsilon)$.

\end{itemize}

\end{theorem}

In order to apply Theorem \ref{pert1} to our model \eqref{discrete} we let
$U=\{ x \in \mathbb{R}^N_+ ; x_1 > 0\}$,
$\Lambda$ be an appropriate subset of $\Gamma$ still to be determined,
and $ f: U \times \Lambda \rightarrow \mathbb{R}^N $,
where $f= (f_j)_{j=1}^{N}$, $x= (x_j)_{j=1}^{N}$,
$\lambda= \gamma= (\gamma _{ij})_{i,j=1}^N$ and

\[
f_j(x,\lambda)= \sum_{i=1}^N  B_i (\bar x) x_i \gamma_{ij} - D_j(\bar x)x_j, \qquad \bar x = \sum_{k=1}^N x_k.
\]

 We also let $$(x_0,\lambda_0)=( K_1 e^1, \I),$$
with $e^1$ denoting the first of the canonical basis vectors of $\R^N$.

\begin{remark}One may notice that $(K_1 e^1)$
mentioned above is not an interior point of $U$.
However, in Theorem \ref{pert1} the assumption that $x_0$ is an interior point of $U$
 is  unnecessarily restrictive. One can use one-sided derivatives
 with respect to some cone or wedge \cite{SmithWalt}.
 Thus, for the model \eqref{discrete} one can use one-sided derivatives of $f$
 with respect to $\mathbb{R}^N_+$.
 \end{remark}

We now have the following theorem describing the dynamics of the model \eqref{discrete}
when mutation is small:

\begin{theorem}
\label{main-finite}
Assume that (A1)-(A3) and (A5), (A6) hold.
Then there exists some $\delta > 0$ such that, for each matrix $\gamma \in \Gamma$
with $\|\gamma - \I \| < \delta$,
 there exists a stable  equilibrium
$x^*_\gamma $ of the ordinary differential equation system \eqref{discrete} with
$x^*_\gamma $ converging to $x_\I^*=
K_1 e^1$ as $ \| \gamma - \I \|  \rightarrow 0$.
Furthermore, if $\|\gamma - \I \|  < \delta $, $x^*_\gamma$
attracts all solutions $x$ of (\ref{discrete}) with $x_1(0) > 0$.
\end{theorem}

{Here $\|\cdot\|$ is any matrix norm for $N\times N$ matrices.}

\begin{proof} First note that $ f$ is continuous.
Moreover,  $ f( x^*_\I ,\I)= 0$
and, by Theorem \ref{re:directed}, $ x^*_\I$ is
 globally attractive for initial measures in $U$.

   Also, observe that the Jacobian matrix $\frac{\p f}{\p x} (x, \gamma)$
   is continuous on $U \times \Gamma$, and
   evaluating it at $ ( x^*_\I ,\I)$ we obtain an upper triangular matrix with elements:

\begin{equation*}
\frac{\p f_j}{ \p x_i} (x^*_\I, \I)
=
\left \{ \begin{array}{ll}
\displaystyle ( B_1'(K_1)
- D_1'(K_1)) K_1, \quad&
 j=1, \; i=1,\dots, N,
 \\
B_j (K_1) - D_j (K_1),  &   i=j=2, \dots, N,
\\
0, & \hbox{ otherwise.}
\end{array} \right.
\end{equation*}
Thus, the eigenvalues of the Jacobian matrix $\frac{\p f}{\p x} ( x^*_\I ,I)$  are given by the
 diagonal elements of this matrix, namely,
$$
\Big ( \frac{\partial B_1}{\partial x_1}(K_1)-\frac{\partial D_1}{\partial x_1}(K_1)\Big ) K_1,
\qquad B_j (K_1) - D_j (K_1), \quad j =2, \ldots, N.
$$
From assumptions (A1)-(A6) it is clear that these eigenvalues are real and negative.

As for (H1), we use Theorem \ref{re:pers-robust}
with $Q = \{1, \ldots, N\}$ and $Q^\di = \{1\}$.
Then there exists some $\tilde \delta > 0$ and some $\check \epsilon >0$
such $\liminf_{t\to \infty} x_1(t) \ge \check \epsilon $ for
all solutions of (\ref{discrete}) with $x_1(0) > 0$.

To satisfy (H1) of Theorem \ref{pert1}, set
$\Lambda = \{\gamma \in \Gamma: \|\gamma - \I\|< \tilde \delta \}$
and
$D = \{y \in \R^N_+; y_1 \ge \check \epsilon, \sum_{j=1}^N y_j \le K_1 +1 \}
\subseteq U $.

\bigskip

All assertions follow from Theorem \ref{pert1} except local stability
of the equilibria $x^*_\gamma$ which follows from the fact that the  Jacobian
matrices $\frac{\p f}{\p x} (x^*_\gamma, \gamma)$ continuously depend on $\gamma$
and the spectral bound of a
matrix continuously depends on the matrix. So all eigenvalues of the Jacobian
matrices $\frac{\p f}{\p x} (x^*_\gamma, \gamma)$  are negative if $\|\gamma - \I \|$
is sufficiently small.
\end{proof}

We turn to the case that there are several strategies for which the
carrying capacity is maximal. We assume that the mutation kernel
is directed to one of those strategies. Without loss of generality,
we assume that the carrying capacity is maximal for the first $m$ strategies
and the mutation matrix is directed to the first strategy.
More precisely,
let $N \ge 2$ and $m \in \{2,\ldots, N\}$
and $K_1 = \cdots =K_m = K^\di$ and $K_j < K^\di$ for $j =m+1, \ldots, N$.
In case that $m =N$, the last inequality is omitted. Then $\gamma^\di \in \Gamma$
is directed to the first strategy if
\begin{equation}
\begin{array}{cl}
\gamma_{11}^\di = 1, \qquad \gamma_{1j}^\di= 0 , \quad & j=2, \ldots, N,
\\
\gamma_{i1}^\di > 0, \qquad & i =1 , \ldots, N,
\\
\gamma_{ij}^\di =0, \qquad & i =1, \ldots, m, \quad j = m+1, \ldots, N.
\end{array}
\end{equation}

\begin{theorem}
\label{main-finite2}
Assume that (A1)-(A3) and  (A6) hold and that $\gamma^\di$ is a mutation
matrix that is directed to the first strategy as just explained.

Then there exists some $\delta > 0$ such that, for each matrix $\gamma \in \Gamma$
with $\|\gamma - \gamma^\di \| < \delta$,
 there exists a stable  equilibrium
$x^*_\gamma $ of the ordinary differential equation system \eqref{discrete} with
$x^*_\gamma $ converging to $x^*_{\gamma^\di}=
K_1 e^1$ as $ \| \gamma - \gamma^\di \|  \rightarrow 0$.
Furthermore, if $\|\gamma - \gamma^\di \|  < \delta $, $x^*_\gamma$
attracts all solutions $x$ of (\ref{discrete}) with $x_1(0) > 0$.
\end{theorem}

\begin{proof}
The proof is the same as for Theorem \ref{main-finite} except for
proving that the Jacobian of $f$ evaluated at $(K^\di e^1, \gamma^\di)$ has
a negative spectral bound. Notice that
\begin{equation}
\begin{split}
f_1( x, \gamma^\di) = & \sum_{i=1}^N B_i(\bar x) x_i \gamma^\di_{i1} - D_1(\bar x) x_1,
\\
f_j(x, \gamma^\di) = & \sum_{i=2}^N B_i(\bar x) x_i \gamma^\di_{i1} - D_j(\bar x) x_j,
\qquad j =2, \ldots, N,
\\
\bar x = & \sum_{k=1}^N x_j.
\end{split}
\end{equation}
It is easy to see that $x^\di = x^*_{\gamma^\di} = K^\di e^1$
is an equilibrium of $f(\cdot, \gamma^\di)$.

We do not need to compute all entries of the Jacobian matrix in order
to determine the sign of its spectral bound,
\begin{equation}
\begin{split}
\frac{\p f_1}{\p x_1} (x^\di, \gamma^\di)= & (B_1'(K^\di) - D_1'(K^\di)) K^\di,
\\
\frac{\p f_j}{\p x_1}  (x^\di, \gamma^\di)= & 0, \qquad j =2, \ldots , N,
\\
\frac{\p f_j}{\p x_j}  (x^\di, \gamma^\di)= & B_j(K^\di) \gamma^\di_{jj} - D_j(K^\di),
\qquad j =2, \ldots , N,
\\
\frac{\p f_j}{\p x_i} (x^\di, \gamma^\di)= & B_i(K^\di) \gamma^\di_{ij} ,
\qquad i, j =2 , \ldots, N,\; i \ne j.
\end{split}
\end{equation}
From these equations, it is apparent that the spectral bound of
$\frac{\p f}{\p x} (x^\di,\gamma^\di)$
is the larger of $ (B_1'(K^\di) - D_1'(K^\di)) K^\di < 0$
and the spectral bound of the matrix
$\big (\frac{\p f_j}{\p x_i} (x^\di,\gamma^\di) \big )_{2 \le i,j\le N} $.
The row sums of this latter matrix are
\[
\sum_{j=2}^N \frac{\p f_j}{\p x_i} (x^\di,\gamma^\di)
=
\sum_{j=2}^N B_i(K^\di) \gamma_{ij}^\di -  D_i(K^\di)
=
B_i(K^\di) - D_i(K^\di)  - B_i(K^\di) \gamma^\di_{i1}, \quad i =2, \ldots, N.
\]
The last expression is negative: for $i=2, \ldots, m$ because $B_i(K^\di) = D_i(K^\di) >0$
and $\gamma^\di_{i1} > 0$; for $i =m+1, \ldots, N$ because $B_i(K^\di) < D_i(K^\di)$.
By \cite[Rem.A.48]{Thi03}, the spectral bound of
$(\frac{\p f_j}{\p x_i} (x^\di,\gamma^\di))_{2 \le i,j\le N} $ is negative and so is
the spectral bound of $\frac{\p f}{\p x} (x^\di, \gamma^\di)$.
\end{proof}


\section{Concluding Remarks}
\label{sec:conclud}

\par We studied the long-time behavior
of measure-valued solutions of a differential equation {that can be viewed
as the model for an evolutionary game or as a selection-mutation population
model}. We first assume that there is a unique fittest
{strategy or trait}. This unique fittest trait is characterized by maximizing the carrying capacity $K(q)$. We {provided conditions under
   which},  in the pure replication case with a unique fittest {trait}, the model converges to a Dirac measure {concentrated} at the fittest trait {provided that
   the fittest trait is contained in the support of the initial measure.}

{One sufficient condition is that this trait (strategy) does not only maximize
the carrying capacity but also the reproduction numbers at all relevant population
densities. Another sufficient condition is that the fittest strategy is isolated
in the strategy space. If there are several fittest traits, the solutions converges
to a set of measures the support of which is contained in the set of fittest traits.}

{If mutations are allowed, we only could get results on the long-time behavior
of solutions if we assumed that the
set of fittest traits was topologically separated from the less fit traits.
Convergence to the Dirac measure concentrated at the fittest trait holds
if the mutations are directed in the sense that there are no mutation losses
from the set of fittest traits  and that among the fittest traits
there is one particular trait  such that mutations within this set are directed towards this particular  trait. }

We also showed that,    {for a discrete strategy space and mutations
that are small or directed}, there is a locally asymptotically stable equilibrium that attracts all solutions with initial conditions that are positive at the fittest strategy.

 The abstract theory presented in this paper finds practical application in epidemic models.
 Epidemic models which consider the dynamics of multi-strain pathogens have been studied in
 the literature
 (e.g., {\cite{AA1,BT, SmTh13}}). These models
have been formulated as systems of ordinary differential equations where infected
individuals are distributed over a set of $n$ classes each carrying a particular strain
of a finite and discrete {trait} (strategy) space. However, often a continuous (strategy)
space is needed. For example, think
of a particular disease that has transmission rate $\beta$ with possible values in the
interval $[\underline \beta,\overline \beta]$. Then infected individuals are distributed
over a {trait} space with transmission taking values in this interval.
The theory presented here will have potential applications in treating
such distributed rate epidemic models with possibly more than one fittest strain; however, at the present state of the theory,
selection of the fittest strains could only be concluded if
mutations are excluded.
\\

 \noindent {\bf Acknowledgements:}  We  thank Paul Salceanu  and Ping Ng
 for helpful discussions.
 Azmy Ackleh and John Cleveland were partially supported by the
 National Science Foundation under grant \# DMS-0718465. Azmy  Ackleh was also supported
 by the National Science Foundation under grant \# DMS-1312963.

\end{document}